\documentclass[12pt]{article}

\usepackage[amsmath]{e-jc_mod}
\usepackage{graphicx}

\newcommand\calA{\mathcal{A}}
\newcommand\calC{\mathcal{C}}

\newcommand\calH{\mathcal{H}}
\newcommand\calL{\mathcal{L}}
\newcommand\calO{\mathcal{O}}
\newcommand\calT{\mathcal{T}}

\newcommand\NN{\mathbb{N}}
\newcommand\QQ{\mathbb{Q}}
\newcommand\RR{\mathbb{R}}
\newcommand\ZZ{\mathbb{Z}}

\newcommand\cL{{\mathcal L}}

\providecommand{\abs}[1]{\left\lvert#1\right\rvert}

\providecommand{\floor}[1]{\left\lfloor#1\right\rfloor}
\providecommand{\ceil}[1]{\left\lceil#1\right\rceil}
\def\vec#1{\mathchoice{\mbox{\boldmath$\displaystyle\bf#1$}}{\mbox{\boldmath$\textstyle\bf#1$}}{\mbox{\boldmath$\scriptstyle\bf#1$}}{\mbox{\boldmath$\scriptscriptstyle\bf#1$}}}

\DeclareMathOperator{\Vol}{Vol}
\DeclareMathOperator{\Kn}{Kn}
\DeclareMathOperator{\lcm}{lcm}

\dateline{Sep 11, 2025}{Jun 19, 2026}{TBD}
\Copyright{The authors. Released under the CC BY-NC-ND license.}

\title{Counting with two-level polynomials}
\author{
Tristram Bogart\authornote{1}
\and
Kevin Woods\authornote{2}
}
\authortext{1}{Departamento de Matem\'aticas, Universidad de los Andes, Bogot\'a, Colombia (\email{tc.bogart22@uniandes.edu.co}).}

\authortext{2}{Department of Mathematics, Oberlin College, Oberlin, Ohio, USA (\email{kwoods@oberlin.edu)}.}

\begin{document}
\maketitle

\begin{abstract}
We examine combinatorial counting functions with two parameters, $n$ and $q$. For fixed $q$, these functions are (quasi-)polynomial in $n$. As $q$ varies, the degree of this polynomial is itself polynomial in $q$, as are the leading coefficients. We carefully define these two-level polynomials, lay out their basic algebraic properties, and provide a schema for showing a function is a two-level polynomial. Using the schema, we prove that a variety of counting functions arising in different areas of combinatorics are two-level polynomials. These include chromatic polynomials for many infinite families of graphs, partitions of an integer into a given number of parts, placing non-attacking chess pieces on a board, Sidon sets, and Sheffer sequences (including binomial type and Appell sequences).
\end{abstract}


\section{Introduction}
Many combinatorial problems come in families parametrized by two natural numbers playing different roles. For example, how many partitions of the integer $n$ are there into $q$ parts?  How many proper colorings with $n$ colors are there of the vertices of the grid graph $P_q\times P_q$? How many Sidon sets are there of length $n$ and size $q$?  How many ways are there to place $q$ queens on an $n \times n$ chessboard such that no two queens attack each other? How about $q$ knights? The $q$-queens problem, analyzed by Chaiken, Hanusa, and Zaslavsky \cite{CHZ} using Ehrhart theory (see \cite{BR}), was our inspiration for this paper.

These counting functions share a few common characteristics. For fixed $q$, they are polynomial in $n$ (or quasi-polynomial, as defined later). If we are interested in the asymptotics for large $n$, we therefore want to compute the leading coefficients of this polynomial. As $q$ changes, these counting functions are roughly exponential in $q$. To be precise, the degrees of these polynomials in $n$ are themselves polynomial in $q$, and we would like to know how the leading coefficients change with $q$. If $d$ is the degree of a polynomial in $n$ and $r\in\NN=\{0,1,2,\ldots\}$, define the \emph{codegree} $r$ term/coefficient to mean the degree $d-r$ term/coefficient. We will show that in all of these families and others, for many or all values of $r$, the codegree $r$ coefficient is a \emph{polynomial} in $q$. We call sequences of functions $f_q(n)$ with this property \emph{two-level polynomials}, to be defined precisely in 
Definitions~\ref{def:two-level} and \ref{def:two-levelQP}.

\begin{example}
\label{ex:PqxPq}
Let $f_q(n)$ be the number of proper vertex colorings of $P_q\times P_q$ with $n$ colors, where $P_q$ is a path with $q$ vertices, the product is the standard Cartesian product (so $P_q\times P_q$ is a square grid), and proper coloring means that no adjacent vertices can be the same color. For fixed $q$, this is the \emph{chromatic polynomial} \cite{RT}, a monic polynomial of degree $\abs{V}=q^2$, with codegree one coefficient ${-\abs{E}}=-2q(q-1)$. Since the graph contains no 3-cycles, the codegree two coefficient is  known to be $\binom{\abs{E}}{2}$. In other words,
\[f_{q}(n) = 1n^{q^2}-\left(2q^2-2q\right)n^{q^2-1}+\left(2q^4-4q^3+q^2+q\right)n^{q^2-2}-\cdots.\]
We will confirm in Section \ref{sect:finite} that (with a caveat to be explored soon), for fixed $r$, the codegree $r$ term is a polynomial in $q$. Fully computing $f_q(n)$  is notoriously difficult \cite{RT}.
\end{example}

We will demonstrate that a wide variety of counting functions are two-level polynomials, and we believe that most of these results are fundamentally new. Furthermore, they show a surprising common structure in apparently unrelated problems. These results do not directly resolve any of the counting problems, but they do open future research directions. In the best case, one would aim to resolve an entire family of counting problems by calculating the sequence of coefficient polynomials via some sort of recursion. Short of this, one could calculate successive approximations, simultaneously valid for all $q$, by calculating the first few of the coefficient polynomials. In the case of quasi-polynomials, it also may be feasible to calculate or bound the periods of the coefficient functions. All of these approaches are pursued for the nonattacking queens problem in \cite{CHZ} and its sequels \cite{CHZ2, CHZ3, CHZ4, CHZ5, CHZ6, CHZ7}, including generalizing to pieces that can move arbitrarily far in certain directions. For fixed-move pieces like knights, for Sidon sets, and for several families of graph polynomials, any such results would be novel.  

Before giving the definitions, we mention two simple examples that illuminate an important subtlety.

\begin{example}
\label{ex:infinite}
Let $f_q(n)=(n+1)^q$. The codegree $r$ coefficient is the polynomial
\[\binom{q}{r}=\frac{1}{r!}(q)_r=\frac{1}{r!}q(q-1)\cdots(q-r+1),\]
so $f_q(n)$ is a two-level polynomial. Note that this works even when $r>q$: the coefficient of the negative power $n^{q-r}$ should be zero, and the polynomial indeed evaluates to zero.

In Section \ref{sect:alg}, we will discuss algebraic properties of two-level polynomials, including that they are preserved under products. Indeed, $\big(f_q(n)\big)^2 = (n+1)^{2q}$ is still a two-level polynomial, with codegree $r$ coefficient $\binom{2q}{r}$.
\end{example}

\begin{example}
\label{ex:depth}
We want to call $f_q(n)=n^q+n^{q-1}+\cdots+1$ a two-level polynomial, since the codegree $r$ coefficient is the polynomial (in $q$) 1, even though this is not valid for $r>q$ (for which the coefficient is 0). We say that this two-level polynomial has \emph{depth function} $q$, because the terms of codegree up to $q$ agree with polynomials in $q$ (to be defined precisely in Definition~\ref{def:two-level}). By contrast, we say Example~\ref{ex:infinite} has \emph{infinite depth}.

This seems a small quibble, since the depth is sufficient to cover all of the actual terms in $f_q(n)$, but note that \[\big(f_q(n)\big)^2 = n^{2q}+2n^{2q-1}+\cdots + qn^{q+1}+(q+1)n^{q}+qn^{q-1} + \cdots + 2n +1,\]
is a two-level polynomial of degree function $2q$ whose depth function remains $q$: only the codegree $r$ coefficients with $r\le q$ agree with the polynomial (in $q$) formula $r+1$.
\end{example}

We will see many examples of natural counting problems (including Example~\ref{ex:PqxPq}) where the depth is less than the degree. We are now ready for some precise definitions.

\begin{definition} A \emph{numerical polynomial} is a polynomial $g(x) \in \QQ[x]$ such that $g(q) \in \ZZ$ for every $q \in \ZZ$.
\end{definition}

For example, $g(x) = \frac{x(x-1)}{2} = \frac{x^2}{2} - \frac{x}{2}$ is a numerical polynomial: although its coefficients are rational, $g(q) = \binom{q}{2}$ is integral for every $q \in \ZZ$.

\begin{definition}
\label{def:two-level}
Let $g$ be an eventually nonnegative numerical polynomial, let $S=\{q\in\NN:\ g(q)\ge 0\}$, and let $e:S \to \NN \cup \{-1,\infty\}$.
\begin{enumerate}
\item[(a)] A \emph{two-level polynomial} of degree function $g=g(q)$ and depth function $e=e(q)$ is an infinite sequence of polynomials $\{f_q\}_{q\in S}$, with $f_q(n)$  of degree $g(q)$, that has the following property: writing \[f_q(n) = \sum_{r=0}^{\infty}c_r(q)n^{g(q)-r},\] where $c_r(q)=0$ if $r>g(q)$, there exist polynomials $\{\phi_r(q)\}_{r\ge 0}$ such that, for every $q\in S$ and every $r \leq e(q)$, we have $c_r(q) = \phi_r(q)$. 
 \item[(b)] For $q\in \NN\setminus S$, define $f_q(n)=0$ and $e(q)=-1$, extending the definition of $f_q(n)$ to all $q\in\NN$.
 \end{enumerate}
\end{definition}

\begin{remark}
\label{rem:neg1}
Depth $-1$ means that not even the leading coefficient $c_0(q)$ needs to agree with the polynomial $\phi_0(q)$, so that $\phi_r(q)$ can be completely unrelated to the value of $f_q(n)$. This convention allows for algebraic manipulation and even infinite sums (cf.\ Proposition~\ref{prop:sequence}) with minimal notation.

For example, for fixed $r\in\NN$, $n^{q-r}$ is only a polynomial when $q\ge r$. But if we define the two level polynomial $f^{(r)}_q(n)$ based on the sequence $\{n^{q-r}\}_{q\ge r}$, then $f^{(r)}_q(n)=0$ for $q<r$, and the infinite sum $\sum_{r=0}^\infty f^{(r)}_q(n)$ makes sense. Indeed, it is the degree $q$, depth $q$ two-level polynomial in Example~\ref{ex:depth}.
\end{remark}

\begin{remark}
 Infinite depth, by contrast, implies that even the infinite number of $r>g(q)$ still have $\phi_r(q)=c_r(q)=0$, as in Example \ref{ex:infinite}.
\end{remark}

\begin{remark} \label{rem:mindepth}
If the definition of two-level polynomial is satisfied for some depth function $e$, then it is also satisfied for any depth function $e'$ such that $e'(q) \leq e(q)$ for every $q$. Whenever we affirm that a certain sequence is a two-level polynomial of depth function $e$, we are not claiming that $e$ is minimal. The degree function, on the other hand, is unique. 
\end{remark}

Many counting problems we will encounter are not quite polynomial in $n$, in that they will also depend on $n\bmod p$, for some period $p$.

\begin{definition}
A function $f:\NN\rightarrow \QQ$ is a \emph{quasi-polynomial} (of degree $d$) if there exist periodic functions $b_r:\NN\rightarrow \QQ$ ($0\le r\le d$, with $b_0$ not identically zero) such that
\[f(n)=\sum_{r=0}^d b_r(n)n^{d-r}.\]
A common period of $b_0,\ldots,b_d$ is called a \emph{period} of $f$.
\end{definition}

\begin{example}
By a partition of $n$ into $q$ parts, we mean a way to write $n$ as a sum of $q$ positive integers, where the order of the summands doesn't matter (so $3$ has only one partition into two parts: $1+2$). Let $f_q(n)$ be the number of partitions of $n$ into $q$ parts.

The partitions of $n$ into two parts are $1+(n-1),\ 2+(n-2),\ \ldots,\ \floor{n/2}+\ceil{n/2}$, so
\[f_2(n)=\floor{\frac{n}{2}} = \frac{1}{2}n+\begin{cases} 0,&\text{if $n=0\bmod 2$,}\\ -\frac{1}{2},&\text{if $n=1\bmod 2$.}\end{cases}\]
This is a quasi-polynomial of period 2. For fixed $q$, $f_q(n)$ is a quasi-polynomial of period $\lcm(1,2,\ldots,q)$ (see \cite{Cim}). For example,
\begin{align*}
p_3(n)&=\frac{1}{12}\left(n^2 + [0,-1,-4,3,-4,-1]_6\right)\text{ and }\\
p_4(n)&=\frac{1}{144}\left(n^3+3n^2+[0,-9]_2n+[0,5,-20,-27,32,-11,-36,5,16,-27,-4,-11]_{12}\right),
\end{align*}
where $[a_0,\ldots,a_{p-1}]_p$ is the periodic function $a_{n\bmod p}$.
\end{example}

Many examples of quasi-polynomials come from Ehrhart theory (see Section \ref{sect:Ehrhart}), and it is typical that the period of the codegree $r$ coefficient $c_r$ grows with $r$. This inspires the following definition of a two-level quasi-polynomial.

\begin{definition}
\label{def:two-levelQP}
Let $g$ be an eventually nonnegative numerical polynomial, $S=\{q\in\NN:\ g(q)\ge 0\}$, $e:S \to \NN \cup \{-1,\infty\}$, and $p:\NN\to\ZZ_+$ such that $p(r)$ divides $p(r+1)$ for all $r$.
\begin{enumerate}
\item[(a)] A \emph{two-level quasi-polynomial} of degree function $g=g(q)$, depth function $e=e(q)$, and period function $p=p(q)$  is an infinite sequence of quasi-polynomials $\{f_q\}_{q\in S}$, with $f_q(n)$ of degree $g(q)$, that has the following property. Writing
\[f_q(n) = \sum_{r=0}^{\infty}c_r(q;n)n^{g(q)-r},\]
where $c_r(q;n)$ is a periodic function of $n$ (for fixed $q$) and $c_r(q;n)=0$ if $r>g(q)$, there exist polynomials $\{\phi_r(q;i)\}_{r\ge 0,\ 0\le i< p(r)}$ such that, for every $n\in \NN$, $q\in S$, $r \leq e(q)$, we have \[c_r(q;n) = \phi_{r}\big(q;\, n\bmod p(r)\big).\] 
 \item[(b)] As before, for $q\in \NN\setminus S$, define $f_q(n)=0$ and $e(q)=-1$.
 \end{enumerate}
\end{definition}

\begin{remark}
Just as we observed about the depth function in Remark \ref{rem:mindepth}), the period function is not unique. If $p$ is a valid period function, then so is any function $p'$ such that $p(r)$ divides $p'(r)$ for every $r$. Again, whenever we affirm that $p$ is a period function we are not claiming that it is minimal. 
\end{remark}

\begin{example} An odd feature of this definition is illustrated by the sequence $f_q(n) = n^{2q} + (-1)^n n^q$. This is not a two-level \emph{polynomial} because, for fixed $q$, it does not yield a polynomial in $n$ but rather a quasi-polynomial. However, it is a two-level \emph{quasi}-polynomial of depth function $q-1$ and coefficient functions $\phi_0(q) = 1$, $\phi_r(q) = 0$ for $r > 1$. That is, the period function $p(r)$ is always one! This phenomenon will recur in the context of integer partitions in Section \ref{sect:gfs}.
\end{example}

In Section~\ref{sect:alg}, we will develop an algebra of two-level polynomials, constructing some basic examples and detailing how to use them as building blocks for more complicated ones.

In Section~\ref{sect:schema}, we develop a schema (Theorem~\ref{thm:schema}) that is useful for proving that counting functions are two-level polynomials, and we illustrate it with an intermediate-level example:  $f_q(n)=n(n+1)\cdots(n+q-1)$, the rising factorial.

In Section~\ref{sect:chromatic}, we apply this schema to the chromatic polynomials of some important families of graphs. First we look at two intermediate-level examples, the complete bipartite graph $K_{q,q}$ and the product of a complete graph and a path $K_q\times P_q$, and we gain comfort with the schema by giving formulas for their two-level chromatic polynomials. Then we analyze more complex examples, such as Kneser graphs $\Kn(q,k)$, Johnson graphs $J(q,k)$, and powers of complete graphs $K_q^k$ (where $k$ is fixed and $q$ is our usual parameter); for all of these, we are not aware of any known formula for their chromatic polynomials. We indicate how the schema can also be applied to characteristic polynomials and matching polynomials of the same families of graphs.

In Section~\ref{sect:Ehrhart} we apply the schema in the context of Ehrhart theory. We first rephrase the proof of the chess queens result (which includes all pieces that move arbitrarily far in straight lines) in \cite{CHZ} in our language. Then we give an analagous proof for the number of (ordered) Sidon sets \cite{Sidon}. That is, we show that both of these counting functions are two-level quasi-polynomials of infinite depth. We indicate what hypotheses appear to be necessary in order to apply the schema to other families of counting problems phrased in Ehrhart theory language. We give a hint at the many ways of generalizing this by also analyzing placements of $q$ white and $q$ black queens on an $n\times n$ board so that no queens of opposite colors are attacking each other.

All of the above examples have infinite depth. In Section~\ref{sect:finite}, we apply this schema to the chromatic polynomial of $P_q^k$ (for fixed $k$, so a $k$-dimensional grid graph) which has degree function $q^k$ and finite depth function $q$. We similarly apply it to fixed powers of the cycle graph $C_q$, for which the depth function is $q-2$, and discuss other possible extensions.

In Section~\ref{sect:gfs}, we examine two examples where generating functions can be used to prove that $f_q(n)$ is a two-level polynomial. First, any Sheffer sequence of polynomials (a tool in umbral calculus \cite{RKO}) is (after multiplying by $c^q$ for some constant $c$) a two-level polynomial of infinite depth; these include binomial type and Appell sequences, and in particular include falling factorials, Bernoulli polynomials, Hermite polynomials, and Touchard polynomials. Second, the number of partitions of $n$ into $q$ parts is (after multiplication by $q!(q-1)!$) a two-level quasi-polynomial of degree function $q-1$, depth function $\floor{(q-1)/2}$, and period function one. 

In Section~\ref{sec:knights}, we answer an open problem from \cite{CHZ}: what about placing non-attacking chess pieces like knights, which have only a finite set of moves? We obtain a two-level polynomial, but must carefully analyze how it is not valid for small $n$.

In Section~\ref{sec:conc}, we point to some open questions and future directions for research.

\section{The algebra of two-level polynomials}
\label{sect:alg}
We lay out some basic algebraic properties of the class of two-level quasi-polynomials that will be applied to combinatorial examples in the following sections. The proofs in this section tend to be intuitive, but calculation intensive. Since a two-level polynomial is simply a two-level quasi-polynomial such that each $f_q$ is a polynomial, these results all apply to two-level polynomials as well.

\begin{proposition} \label{prop:bivariate}
Let $f(q;n)$ be a function that is polynomial in $q$, and also a quasi-polynomial in $n$ of degree $d$ and period $p_0$. Then $f_q(n) = f(q;n)$ is a two-level quasi-polynomial of constant degree function $g(q) = d$, infinite depth, and constant period function $p(r)=p_0$.\end{proposition}

\begin{proof} 
Given such a function $f(q;n)$, we can write
  \[ f_q(n) = f(q;n) = \sum_{r=0}^{d} a_r(n) \left( \sum_{\ell=0}^{\delta_r} b_{r, \ell}q^\ell \right) n^{d-r}\]
where $a_0(n), \dots, a_d(n)$ are periodic functions of common period $p_0$ and $\delta_0, \dots, \delta_d$ are natural numbers. 
    This sequence satisfies the definition of a two-level quasi-polynomial of constant degree $g(q)=d$ and infinite depth by taking $p(r)=p_0$ for every $r$ and   
    \[ \phi_r(q;i) = \begin{cases} a_r(i) \left( \displaystyle\sum_{\ell=0}^{\delta_r} b_{r, \ell}q^\ell \right), & \text{if $r\le d$,} \\
      0, & \text{ otherwise.} \end{cases} \qedhere \]
\end{proof}
    
  In particular, we typically start with (quasi-)polynomials in just one of the variables $q$ or $n$, which are immediately seen to be two-level polynomials:
\begin{corollary} \label{cor:constantdegree}
\ 
      \begin{enumerate}
      \item[(a)] Let $f(q)$ be a polynomial in $q$. Then the sequence $f_q(n) = f(q)$ is a two-level polynomial of degree function zero and infinite depth.
      \item[(b)] Let $f(n)$ be a quasi-polynomial in $n$ of degree $d$ and period $p_0$. Then the sequence $f_q(n) = f(n)$ is a two-level quasi-polynomial of constant degree function $g(q)=d$, infinite depth, and constant period function $p(r) = p_0$. 
      \end{enumerate}
     \end{corollary}

The most straightforward arithmetic operation that we can apply to two-level polynomials is the product. The degree, depth, and period functions behave as one might expect.
 \begin{proposition} \label{prop:products}
   Let $f_q(n)$ and $f'_q(n)$ be two-level quasi-polynomials
   of respective degree functions $g(q)$ and $g'(q)$, depth functions $e(q)$ and $e'(q)$, and period functions $p(r)$ and $p'(r)$. Then the product $f_q(n)\cdot f'_q(n)$ is a two-level quasi-polynomial of degree function $g + g'$, depth function $\min(e,e')$, and period function $\lcm(p, p')$.
 \end{proposition}

 \begin{proof} Write
 \[f_q(n) = \sum_{r=0}^{\infty}c_r(q;n)n^{g(q)-r}\quad\text{and}\quad f'_q(n) = \sum_{r=0}^{\infty}c'_r(q;n)n^{g(q)-r},\]
 and let $\{\phi_r(q;i)\}_{r \geq 0,\ 0 \leq i < p(r)}$ and $\{\phi'_{r}(q;i)\}_{r \geq 0,\ 0 \leq i < p'(r)}$ be the respective coefficient polynomials, as in Definition~\ref{def:two-levelQP}. The product is given by
 \begin{align*} f_q(n)f'_{q}(n) & = \sum_{s=0}^{\infty}\sum_{s'=0}^{\infty} c_s(q;n)n^{g(q)-s} c'_{s'}(q;n) n^{g'(q)-s'} \\
& = \sum_{r=0}^{\infty} \sum_{s=0}^r c_s(q;n) c_{r-s}'(q;n) n^{g(q)+g(q')-r}
      \end{align*}
which is a quasi-polynomial in $n$ of degree function $(g+g')$. For any $r \leq \min\big(e(q),e'(q)\big)$, the coefficient of $n^{g(q)+g'(q)-r}$ is
   \begin{align*} \sum_{s=0}^r c_s(q;n) c_{r-s}'(q;n) & = \sum_{s=0}^r \phi_{s}\big(q;\, n\bmod p(s)\big) \phi'_{r-s}\big(q;\, n\bmod p'(r-s)\big) \\
     & = \sum_{s=0}^r \phi_{s}\Big(q;\ n\bmod \lcm\big(p(r),p'(r)\big)\Big) \phi'_{r-s}\Big(q;\ n\bmod \lcm\big(p(r),p'(r)\big)\Big),
   \end{align*}
       where we use that $p(s)$ divides $p(r)$ and $p'(s-r)$ divides $p'(r)$ for $0 \leq s \leq r$. By taking coefficient polynomials 
   \[  \psi_{r}(q;i) := \sum_{s=0}^r \phi_{s}(q;i) \phi'_{r-s}(q;i), \]
   it follows that the product is a two-level quasi-polynomial of depth function $\min(e,e')$ and period function $\lcm(p,p')$.
 \end{proof}

We'll need the following technical lemma later: if one of the two-level polynomials in the product is simply a polynomial in $q$, then we can do slightly better than Proposition~\ref{prop:products} in terms of the depth.

\begin{lemma}
\label{lem:poly_prod}
Let $h(q)$ be a polynomial in $q$ and $f_q(n)$ be a two-level quasi-polynomial of depth function $e(q)$. Then $h(q)f_q(n)$ has depth function
\[\begin{cases} \infty,&\text{if $h(q)=0$,}\\ e(q),&\text{else.} \end{cases}\]
\end{lemma}
\begin{proof}
Write $f_q(n) = \sum_{r=0}^{\infty}c_r(q;n)n^{g(q)-r}$, and let $\{\phi_r(q;i)\}$ be the coefficient polynomials. Then $h(q)f_q(n)=\sum_{r=0}^{\infty}h(q)c_r(q;n)n^{g(q)-r}$ and the coefficient polynomials are $\{h(q)\phi_r(q;i)\}$. If $h(q)=0$, then
\[h(q)c_r(q;n) = 0 = h(q)\phi_{r}\big(q;\, n\bmod p(r)\big)\]
regardless of $r$, that is, the depth is infinite. The remainder follows from Corollary~\ref{cor:constantdegree} and Proposition~\ref{prop:products}.
\end{proof}

 We will also need to take polynomial (in $q$) powers of two-level polynomials. There is an immediate problem here, suggested by the example of $(2n)^q = 2^qn^q$ which is \emph{not} a two-level polynomial since its leading coefficient is exponential in $q$. However, the leading coefficient turns out to be the only problem. If the base two-level polynomial is monic (that is, if the leading coefficient is the constant 1), then we still get a two-level polynomial. If the leading coefficient is a different constant $c$, then we get a two-level polynomial after normalization.

\begin{proposition} \label{prop:powers}
 Let $h(q)$ be a numerical polynomial and $f_q(n)$ be a two-level quasi-polynomial of degree function $g(q)$, depth function $e(q)$, and period function $p(r)$. Suppose the leading coefficient of $f_q(n)$ is a constant $c$. Then  $\left\{f_q(n)^{h(q)}/c^{h(q)}\right\}_{g(q),h(q)\ge 0}$ is a monic two-level quasi-polynomial of degree function $gh$, depth function $e$, and period function $p$.  
\end{proposition}

\begin{proof}
  Write $f_q(n) = \sum_{r=0}^{\infty} c_r(n, q)n^{g(q)-r}$ with $c_0(q;n)=c$, and let $\{\phi_r(q;i)\}$ be the coefficient polynomials. Thus
  \begin{align*}
    \frac{f_q(n)^{h(q)}}{c^{h(q)}} & = \sum_{s=0}^{h(q)} \binom{h(q)}{s} \left(c n^{g(q)} \right)^{h(q)-s} \left( \sum_{r=1}^{g(q)} c_r(q;n)n^{g(q)-r }\right)^s \cdot \frac{1}{c^{h(q)}} \\
     & =  \sum_{s=0}^{h(q)} c^{-s}\binom{h(q)}{s} n^{g(q)h(q)-sg(q)} \sum_{\ell_1 + \dots + \ell_{g(q)}=s} \binom{s}{\ell_1, \dots, \ell_{g(q)}} \prod_{j=1}^{g(q)} \left(c_j(q;n)n^{g(q)-j}\right)^{\ell_j}.
  \end{align*}
 For each $q$, this is a monic quasi-polynomial of degree $gh$. For fixed $r \geq 1$, we obtain terms of codegree $r$ by choosing $s \geq 1$ and $\ell_1, \dots, \ell_{g(q)} \geq 0$ summing to $s$ such that
  \[ g(q)h(q) - r = g(q)h(q)-sg(q) + \sum_{j=1}^{g(q)} \left(g(q) - j \right) \ell_j \]
  which reduces after cancellation to
\[ \sum_{j=1}^{g(q)} j \ell_j = r.\]
  In particular, $\ell_j = 0$ whenever $j > r$. Furthermore,
  \[ s = \sum_j \ell_j \leq \sum_j j \ell_j = r. \]
  Since both $s$ and $j$ are bounded by $r$ in the nonzero terms of codegree $r$, we see that the coefficient of $n^{g(q)h(q)-r}$ is
  \[ \sum_{s=0}^{r} c^{-s} \binom{h(q)}{s} \sum_{\ell_1 + \dots + \ell_r = s} \binom{s}{\ell_1, \dots, \ell_r}\prod_{j=1}^r \left( c_j(q;n) \right)^{\ell_j}. \]
  Whenever $r \leq e(q)$, we have $c_j(n, q) = \phi_{j}\big(q;\,n \bmod p(j)\big) = \phi_{j}\big(q;\,n \bmod p(r)\big)$ for $j=1,\dots,r$ and this coefficient becomes
  \[ \sum_{s=1}^{r} \binom{h(q)}{s} \sum_{\ell_1 + \dots + \ell_r = s} \binom{s}{\ell_1, \dots, \ell_r}\prod_{j=1}^r \Big( \phi_{j}\big(q;\, n \bmod p(r)\big) \Big)^{\ell_j} \]
 which is a polynomial in $q$ for each residue class of $n$ mod $p(r)$. The proof follows.   
\end{proof}

Since quasi-polynomials in $n$ are two-level quasi-polynomials (Corollary~\ref{cor:constantdegree}), we obtain the following useful corollary.

\begin{corollary} \label{cor:powers}
      If $f$ is a quasi-polynomial of degree $d$, period $p_0$, and constant leading coefficient $c$, and if $h$ is a numerical polynomial, then $\left\{f(n)^{h(q)} / c^{h(q)}\right\}_{h(q)\ge 0}$ is a two-level quasi-polynomial of degree function $g(q) = d \cdot h(q)$, infinite depth, and  constant period function $p(r)=p_0$. 
    \end{corollary}

Next, we consider sums of two-level polynomials. Here there are two potential problems. One is that the degree and depth functions of the sum of two-level polynomials are unpredictable because of the possibility that the leading terms cancel. The other is exemplified by the case of the sum of $f_q(n) = n^q$ and $f'_q(n)=1$. Both have infinite depth, and they are of different degrees so their leading terms do not cancel, but their sum $n^q+1$ has depth function only $q-1$. Note that this isn't a problem if their degree functions differ by a constant: for example, $n^{q+3}+n^q$ still has infinite depth.

There are many precise statements we could make about addition (all with easy proofs). We will state a useful one about \emph{infinite} sums of two-level polynomials. As well as avoiding cancellation of leading terms and non-constant degree shifts, we want these sums to make sense formally; that is, we should never need to add infinitely many terms of the same degree. By choosing the sequence of degree functions appropriately, we deal with all of these issues.

\begin{proposition} \label{prop:sequence} Let $g(q)$ be a nonnegative numerical polynomial and $0 = d_0 < d_1 \leq d_2 \leq \cdots$ be a sequence of natural numbers such that for each $d \in \NN$, the set $I_d = \{t: \, d_t = d\}$ is finite. Let $f^{(0)}_q(n), f^{(1)}_q(n), f^{(2)}_q(n), \dots$ be a sequence of two-level quasi-polynomials such that $f^{(t)}_q(n)$ has degree function $g(q) - d_t$, depth function $e^{(t)}(q)$, and period function $p^{(t)}(r)$.  Then
  $\sum_{t=0}^{\infty} f^{(t)}_q(n)$ is a two-level quasi-polynomial of degree function $g$, depth function $e$, and period function $p$ where
  \[ e(q) = \min_{t\in\NN} \left(e^{(t)}(q) + d_t\right)\quad \text{and} \quad p(r) = \lcm_{t: \, d_t \leq r} \left(p^{(t)}(r-d_t)\right).\]
 In particular, if $f^{(t)}_q(n)$ has infinite depth for every $t$, then so does the sum. 
\end{proposition}

 \begin{proof}
   By the definition of two-level quasi-polynomials, for each $t$ there are  polynomials $\{\phi_{r}^{(t)}(q;i)\}_{r\ge 0,\ 0\le i< p^{(t)}(r)}$ such that if we write
 \[ f_q^{(t)}(n) = \sum_{s=0}^{g(q)-d_t} c_s^{(t)}(q;n)n^{g(q)-d_t-s},\]
 then $c_s^{(t)}(q;n) = \phi_{s}^{(t)}(q;i)$ whenever $s \leq e^{(t)}(q)$ and $n \equiv i \; \bmod p^{(t)}(s)$. 
 
 Now for each $q$, there are only finitely many $t$ such that $g(q)-d_t \geq 0$. That is, $f_q^{(t)}(n)$ is defined to be zero (see Definitions~\ref{def:two-level}(b) and \ref{def:two-levelQP}(b)) for all sufficiently large $t$, and thus we can define a quasi-polynomial with an infinite sum $f_q(n) := \sum_{t=0}^\infty f_q^{(t)}(n)$. The hypothesis $d_0 < d_1$ implies that the leading term of $f_q^{(0)}(n)$ cannot be cancelled, so the degree of $f_q$ is exactly $g(q)$. Furthermore,
      \begin{align*}
        f_q(n) & = \sum_{t=0}^\infty f_q^{(t)}(n) \\
        & = \sum_{t=0}^\infty \sum_{s=0}^{g(q)-d_t} c_s^{(t)}(q;n)n^{g(q)-d_t-s} \\
        & = \sum_{d=0}^{g(q)} \sum_{t \in I_d} \sum_{s=0}^{g(q)-d} c_s^{(t)}(q;n)n^{g(q)-d-s} \\
        & = \sum_{r=0}^{g(q)} \sum_{d=0}^r \sum_{t \in I_d} c_{r-d}^{(t)}(q;n)n^{g(q)-r}
      \end{align*}
      where we take $r=d+s$ for the last equality.

      Since $f^{(t)}$ has depth function $e^{(t)}$, we have $c_{r-d}^{(t)}(q;n) = \phi_{r-d}^{(t)}(q;i)$ whenever $n \equiv i \; \bmod p^{(t)}(r)$ and $r-d \leq e^{(t)}(q)$. The first condition implies that $n \equiv i \; \bmod p(r)$ and the second is equivalent to $r \leq e^{(t)}(q) + d = e^{(t)}(q)+d_t$. Thus for each $n \equiv i \; \bmod p(r)$ and $r \leq e^{(t)}(q) + d_t\le e(q)$, the coefficient of $n^{g(q)-r}$ is  $\sum_{d=0}^r \sum_{t \in I_d} \phi_{r-d}^{(t)}(q;i)$ which is a polynomial in $q$. That is, a period function of $f_q(n)$ is $p(r)$ and a depth function is $e(q)$. 
 \end{proof}

 Finally, we consider compositions of two-level polynomials. In general, this doesn't work well, or at least there aren't nice bounds on the depth. For example, $f_q(n)=n^{q+1}+n^{q}$ and $h_q(n)=n^q$ both have infinite depth, but
\[f_q\big(h_q(n)\big)=n^{q^2+q}+n^{q^2}\]
only has depth function $q$. However, we get a nice bound if $h_q(n)=h(n)$ is simply a polynomial in $n$.
 
\begin{proposition} \label{prop:evaluation}
If $f_q(n)$ is a two-level quasi-polynomial of degree function $g(q)$, depth function $e(q)$ and period function $p(r)$, and if $h$ is a polynomial of degree $d > 0$ and leading coefficient $c$, then $f_q\big(h(n)\big)/c^{g(q)}$ is a two-level quasi-polynomial of degree function $d 
 \cdot g(q)$, depth function $d \cdot e(q)+d-1$ and period function $p \big( \floor{ r/d } \big)$.  
\end{proposition}

\begin{proof}
  Write $f_q(n) = \sum_{r=0}^\infty c_r(q;n) n^{g(q)-r}$ and let $\{\phi_r(q;i)\}$ be the coefficient polynomials. Note that $c_r(q;n)$ is a two-level quasi-polynomial of degree 0, period the constant $p(r)$, and depth function
 \[e_r(q):=\begin{cases}
 \infty,&\text{if $r\le e(q)$,}\\
-1,&\text{if $r>e(q)$,}
 \end{cases}\]
 since in the first case it agrees with the polynomial $\phi_{r}\big(q;\,n\bmod p(r)\big)$.

We have
 \[\frac{f_q\big(h(n)\big)}{c^{g(q)}}  = \sum_{r=0}^\infty c_r(q;n) c^{-r} \frac{h(n)^{g(q)-r}}{c^{g(q)-r}}, \]  
 where, by Proposition~\ref{prop:powers}, $h(n)^{g(q)-r}/c^{g(q)-r}$ is a two-level quasi-polynomial of degree $dg(q)-dr$, period one, and infinite depth. Then, by Proposition~\ref{prop:products}, the index $r$ term in the sum is a two-level quasi-polynomial of degree function $dg(q)-dr$, period function $p(r)$ and depth function $e_r(q)$. Finally, Proposition~\ref{prop:sequence} gives us that $f_q\big(h(n)\big)/c^{g(q)}$ is a two-level quasi-polynomial of degree function $d \cdot g(q)$, depth function
\[\min_r\big(e_r(q)+dr\big)=\min_{r:\, r> e(q)}\left(-1+dr\right) = -1+d\big(e(q)+1\big) = d\cdot e(q)+d-1,\]
and period function
\[\lcm_{s:\,ds\le r} \big(p(s)\big)=p \big( \floor{r/d} \big),\]
as desired.
 \end{proof}

\section{Some examples and a schema}
\label{sect:schema}
We work through a few examples that illustrate a schema  (Theorem~\ref{thm:schema}) for proving that a counting function is a two-level polynomial. This schema applies to a large variety of problems.

\subsection{Two examples}
Let's start with a simple example from the binomial theorem.

\begin{example}
We will show that $f_q(n)=(n+1)^q$ is a two-level polynomial of degree function $q$ and infinite depth (which is already guaranteed by Corollary~\ref{cor:powers}, but we will recover the usual explicit sum).
\begin{itemize}
\item Write our function as a \emph{finite sum of known two-level polynomials}, over some sort of combinatorial \emph{objects}:
\[f_q(n) = \sum_{A\subseteq [q]}n^{q-\abs{A}}\cdot 1^{\abs{A}},\]
where the object $A$ is the subset of the binomials from which a 1 rather than an $n$ is chosen. For fixed $A$, $\left\{n^{q-\abs{A}}\right\}_{q\ge \abs{A}}$ is certainly a two-level polynomial. This example is a positive sum, but later we will employ inclusion-exclusion.
\item \emph{Partition the terms into types}, such that \emph{all terms of the same type are identical}. In this case, all $A$ of the same cardinality are defined to be of the same type.
\item Show that the \emph{number of terms of each type is a polynomial in $q$}: there are
\[\binom{q}{a}=\frac{1}{a!}(q)_a=\frac{1}{a!}q(q-1)\cdots(q-a+1)\]
subsets of cardinality $a$; this polynomial is valid for all $q\ge 0$, evaluating to zero when $q < a$, which will enable us to get a two-level polynomial of infinite depth. In Section~\ref{sect:finite}, we'll see examples that are invalid for small $q$, leading to finite depth.
\item As long as a finite number of types yield terms of a fixed codegree (in this case, one type per codegree), Proposition~\ref{prop:sequence} guarantees that our sum is a two-level polynomial. In this case, we reproduce the binomial theorem,
\[(n+1)^q = \sum_{a=0}^{\infty}\binom{q}{a}n^{q-a}.\]
\end{itemize}
\end{example}
Let's look at an intermediate-level example, which will allow us to refine the step of partitioning the terms of the sum into types.

\begin{example}
\label{ex:rising_fact}
We will compute the rising factorial, $f_q(n)=n(n+1)\cdots(n+q-1)$, as a two-level polynomial of infinite depth (and the falling factorial as well, since its terms are the same up to a factor of $\pm 1$). Our numbering below mirrors the precise statement of the schema in the next subsection, and we make this example slightly more complicated than it needs to be in order to be able to generalize more cleanly.

\begin{enumerate}
\item[(S1)] 
We start by writing $f_q(n)$ as a finite sum, using the well-known expansion in terms of unsigned Stirling numbers of the first kind:
\[f_q(n)=\sum_{\sigma\in S_q}n^{N(\sigma)},\]
where $\sigma$ is a permutation of $[q]$ and $N(\sigma)$ is its number of cycles. We break the process of dividing the \emph{objects} $\sigma$ into types, as follows.

\item[(S2)] Identify the \emph{crux} of each $\sigma$; this is the key piece of the object that, together with $q$,  completely determines its contribution to the sum. In our case, the crux is the set of cycles of $\sigma$ of length at least two. Note that there is a natural set of \emph{isomorphisms} of cruxes based on their cycle types, and that the cycle type and $q$ indeed determine the contribution $n^{N(\sigma)}$.
  
\item[(S3)]  Note that the identity permutation uniquely gives the highest degree term in this sum, $n^q$.

If a crux contains $y$ elements of $[q]$ in $z$ cycles, then $\sigma$ has  $q-y$ fixed points, and the \emph{codegree} of its term is
\[r(\sigma):=q-N(\sigma) = q - \big(z+(q-y)\big) = y-z.\]
In particular, $r(\sigma)$ depends on the isomorphism class of the crux but is independent of $q$. We will need to count the number of cruxes in an isomorphism class as a function of $q$. 

\item[(S4)] There is a natural notion of size of a crux: $y$, the number of elements of $[q]$ that it contains. Since each cycle in a crux has at least two elements,
\[r(\sigma) = y - z \ge y - \frac{y}{2} = \frac{y}{2},\]
and so the size of a crux is bounded by $2r(\sigma)$. This illustrates the remaining key property of the crux: fixed codegree terms come from fixed sized cruxes, independent of $q$. 

\item[(S5)] We could now count the number of cruxes in each isomorphism class directly, but in other examples it will be helpful to further refine these classes into crux \emph{types}, which will be easier to count.

To do this, map the crux to a \emph{foundation}. In this case, the foundation of $\sigma$ will be the subset of elements of $[q]$ in its crux. The size of the foundation (in this case) is exactly the same as the size of the crux. Note that the foundation loses important information about the crux and its contribution to the sum. In particular, in this case, it loses information about the cycle type of $\sigma$, and so different cruxes mapping to the same foundation may have different contributions to the sum. 

\item[(S6)] So why create foundations? Because we have simpler notions of foundation isomorphisms, and it will generally be easier to count the number of foundations in a given foundation isomorphism class (versus cruxes in a given crux isomorphism class). In our case, the foundation is just a subset of $[q]$, and there are obvious options for isomorphisms: define two subsets to be isomorphic if there is an order-preserving bijection between them. In our case, there is exactly one isomorphism class of each foundation size, and in general, we need to make sure that there are finitely many.

\item[(S7)] Isomorphisms of foundations immediately induce isomorphisms of cruxes, but note that not all crux isomorphisms are induced in this way. For example, $(123)$ and $(496)$ are isomorphic cruxes, but there is no foundation isomorphism $\{1,2,3\}\rightarrow\{4,6,9\}$ inducing this, because we defined foundation isomorphisms to be order-preserving. We define crux \emph{types} to be the equivalence classes of cruxes induced by foundation isomorphisms. That is, $(132)$ and $(496)$ are of the same type, but $(123)$ is a different type. In particular, the crux types refine the crux isomorphism classes.

\item[(S8)] In general, we want the number of crux types corresponding to an isomorphism class of foundations to be finite, and it clearly is here. In particular, if a foundation has $a$ elements, then each crux type can be represented by a derangement of $[a]$, and there are only finitely many of these.

\item[(S9)] Finally, we now count the number of cruxes of a given type. Given the foundation $[a]$, there are $\binom{q}{a}$ foundation isomorphisms that map $[a]$ into $[q]$. If $\tau$ is a derangement of $[a]$, then applying each foundation isomorphism to $\tau$ induces a unique isomorphic crux with foundation in $[q]$, and therefore there are $\binom{q}{a}$ cruxes of this type. Summing gives us a two-level polynomial of infinite depth:
\[f_q(n)= \sum_{a=0}^{\infty}\sum_{\tau \text{ deranges }[a]}\binom{q}{a}n^{q-a+N(\tau)}.\]
Note that $N(\tau)\le a/2$, so there are a finite number of terms of each codegree.

\end{enumerate}
\end{example}

In Sections~\ref{sect:chromatic} and \ref{sect:Ehrhart}, the foundation will always be something simple, such as a set $[a]$, yielding two-level polynomials of infinite depth. In Section~\ref{sect:finite}, we will see more complicated foundations, yielding finite depth.

\subsection{The Schema}
We now state the schema precisely, mirroring the numbering from the previous example. We will start Section~\ref{sect:chromatic} with intermediate-level examples further illustrating this schema.
\begin{theorem} \label{thm:schema}
Suppose $f_q(n)$ is such that
\begin{enumerate}
\item[(S1)] For $q\in \NN$, we can write
\[f_q(n)=\sum_{O\in\calO_q} p^{O}(n),\]
where $\calO_q$ is finite and $p^{O}(n)$ is a (quasi-)polynomial. We call the $O\in \calO_q$ {\bf objects}.

\item[(S2)] There is a map $\alpha$ from objects to {\bf cruxes}, bijective for each $q$, and there exist {\bf isomorphisms} of cruxes across $q$.  Given an isomorphism class $\calC$ of cruxes, there is a two-level (quasi-)polynomial $f_q^{\calC}(n)$ such that
\begin{itemize}
\item $f_q^{\calC}(n)$ is infinite depth for $q$ such that $\alpha(\calO_q)\cap\calC$ is nonempty, and
\item if $O\in\calO_q$ is such that $\alpha(O)\in {\calC}$, then
\[p^{O}(n) = f_q^{\calC}(n).\]
\end{itemize}

\item[(S3)] There is a numerical polynomial $g(q)$ such that for each isomorphism class $\calC$ of cruxes, the degree function of $f_q^{\calC}(n)$ is $g(q) - r(\calC)$ for some constant {\bf codegree }$r(\calC) \geq 0$. There is a unique class $\calC_0$ for which $r(\calC_0)=0$. 

\item[(S4)] There is a notion of (nonnegative integer) size of a crux $\calC$, respected by the isomorphisms, which is bounded by a function of the codegree $r(\calC)$.

\item[(S5)] There is a map (not necessarily injective) from cruxes to {\bf foundations}. There is a notion of (nonnegative integer) size of a foundation which is bounded by a function of the size of any associated crux.
\item[(S6)] There exist {\bf isomorphisms} of foundations that respect size, such that the number of isomorphism classes with a given foundation size is finite.

\item[(S7)] Isomorphisms of foundations {\bf induce isomorphisms} of cruxes: given a foundation isomorphism $\sigma: A_1\rightarrow A_2$ and a crux $C_1$ associated to $A_1$, we can induce a crux isomorphism $\overline{\sigma}: C_1\rightarrow C_2$ with $C_2$ associated to $A_2$.  Furthermore, $\overline{\sigma^{-1}}=\left(\overline{\sigma}\right)^{-1}$, and if $\sigma_2$ can be composed with $\sigma_1$, then $\overline{\sigma_2 \circ \sigma_1} = \overline{\sigma_2} \circ \overline{\sigma_1}$.  

  These induced crux isomorphisms are a subset of all crux isomorphisms, and so they create equivalence classes of cruxes that refine crux isomorphism classes. We call these new equivalence classes {\bf types}. 
  
\item[(S8)] The number of crux types corresponding to an isomorphism class of foundations is finite.

\item[(S9)] There is a depth function $e:\NN \to \NN \cup \{-1,\infty\}$ with the following property:

Given a foundation $A$, let $\psi_A(q)$ be the number of foundation isomorphisms with $A$ as its domain. Let $r$ be the minimum codegree for a crux associated to $A$. Then $\psi_A(q)$ agrees with a polynomial for all $q$ with $e(q)\ge r$. (In particular, if $e(q)=\infty$ for all $q\in\NN$, then $\psi_A(q)$ is a polynomial for all $q\in\NN$.)
\end{enumerate}
Then $f_q(n)$ is a two-level (quasi-)polynomial of degree function $g$ and depth function $e$.
\end{theorem}

First we use (S7) to prove a lemma that ties (S9) to what we want to count: cruxes of a given type.

\begin{lemma}
\label{lem:automorph}
Let $C$ be a crux of type $T$ and $A$ its foundation, let $\omega_T(q)$ be the number of cruxes of type $T$, and let $c_T$ be the number of automorphisms of $A$ that induce an automorphism of $C$. Then $\omega_T(q) = \psi_A(q)/c_T$.
\end{lemma}

\begin{proof}
Let $X$ be the set of automorphisms of $A$ that induce an automorphism of $C$. Let $C'$ be another crux of type $T$ and $A'$ be its foundation. Let $Y$ be the set of foundation isomorphisms from $A$ to $A'$ that induce an isomorphism from $C$ to $C'$. We will show that $X$ and $Y$ are in bijection. The claim follows immediately, because of the $\psi_A(q)$ isomorphisms with domain $A$, exactly $c_T=\abs{X}$ of them will induce maps from $C$ to any $C'$ of type $T$.

  Indeed, fix $\sigma \in Y$. For each $\tau \in Y$, the composition $\tau^{-1} \circ \sigma$ is a foundation automorphism of $A$, and it follows from (S7) that
  $\overline{\tau^{-1} \circ \sigma} = \left( \overline{\tau} \right)^{-1} \circ \overline{\sigma}$
  is a foundation-induced crux automorphism of $C$. Similarly, for each $\rho \in X$, the composition $\sigma \circ \rho^{-1}$ is a foundation isomorphism from $A$ to $A'$, and it follows from (S7) that $\overline{\sigma \circ \rho^{-1}} = \overline{\sigma} \circ \left( \overline{\rho} \right)^{-1}$ is a foundation-induced crux isomorphism from $C$ to $C'$.

  Define functions $f: Y \to X$ and $g: X \to Y$ by
  \[ f(\tau) = \tau^{-1} \circ \sigma, \quad g(\rho) = \sigma \circ \rho^{-1}. \]
  For $\rho \in X$ we have
  \[ (f \circ g)(\rho) = f \left(\sigma \circ \rho^{-1} \right) = \left( \sigma \circ \rho^{-1} \right)^{-1} \circ \sigma = \rho \circ \sigma^{-1} \circ \sigma = \rho \]
  and for $\tau \in Y$ we have
  \[ (g \circ f)(\tau) = g \left(\tau^{-1} \circ \sigma \right) = \sigma \circ \left( \tau^{-1} \circ \sigma \right)^{-1} = \sigma \circ \sigma^{-1} \circ \tau = \tau. \]
  We conclude that $f$ and $g$ are mutual inverses, and the claim follows.
\end{proof}

\begin{proof}[Proof of Theorem~\ref{thm:schema}]
  By (S1) and (S2), we can write
  \[ f_q(n) = \sum_{O \in \calO_q} p^{O}(n) = \sum_{\calC} \sum_{\substack{O \in \calO_q: \\ \alpha(O) \in \calC}} f_q^\calC(n).\]
  By (S3) we can refine the sum according to the codegree $r$ to obtain 
  \begin{equation} f_q(n) = \sum_{r=0}^\infty \sum_{\calC: \, r(\calC) = r} \sum_{\substack{O \in \calO_q: \\ \alpha(O) \in \calC}} f_q^\calC(n). \label{eq:fromS3} \end{equation}
  
  Note that (S4) and (S5) together imply that there is an upper bound on the foundation size that is a function of the codegree. So given $r$, (S6) implies that there is a finite set $S_r$ of isomorphism classes of foundations that might contribute in codegree $r$. By (S7), we can partition each class $\calC$ of cruxes into types according to the isomorphism class of the foundation. For each $\calA \in S_r$, let $\calT_{\calA,r}$ be the set of types of cruxes of foundation class $\calA$ and codegree $r$. By (S8), $\calT_{\calA,r}$ is finite.  Using the bijection of objects and cruxes (S2) for each $q$, (\ref{eq:fromS3}) can be written as 
\[ f_q(n) = \sum_{r=0}^\infty \sum_{\calA \in S_r} \sum_{T \in \calT_{\calA,r}} \omega_T(q) f_q^{T}(n).\]
 
By Lemma~\ref{lem:automorph}, we now have
\begin{equation} f_q(n) = \sum_{r=0}^\infty \sum_{\calA \in S_r} \sum_{T \in \calT_{\calA,r}} \frac{1}{c_T}\psi_\calA(q)f_q^{T}(n). \label{eq:fromlemma} \end{equation}

Let $r\ge 0$, $\calA\in S_r$, and $T\in \calT_{\calA,r}$ be given. By (S2), $f_q^{T}(n)$ has infinite depth for all $q$ such that $\alpha(\calO_q)\cap T$ is nonempty, that is, for all $q$ such that $\psi_\calA(q)>0$. By (S9), $\psi_\calA(q)$ agrees with a polynomial for all $q$ with $e(q)\ge r$, and so by  Lemma~\ref{lem:poly_prod}, $\psi_\calA(q)f_q^{T}(n)$ also has infinite depth for all $q$ with $e(q)\ge r$ such that $\psi_\calA(q)=0$. In other words, $\frac{1}{c_T}\psi_\calA(q) f_q^{T}(n)$ is a two-level polynomial of degree function $d(q)-r$ and depth function
\[e_r(q):=\begin{cases} -1, & \text{if $e(q)\le r-1$,}\\ \infty, & \text{if $e(q)\ge r$.}\end{cases}\] To apply Proposition \ref{prop:sequence}, note that, since $S_r$ and $\calT_{\calA,r}$ are always finite, we have a finite sum of two-level polynomials of degree function $g(q)-r$ for each $r \geq 0$. Furthermore, by (S3) the sum in (\ref{eq:fromlemma}) begins with a unique two-level polynomial of degree function $g(q)$:
\[\Big|\{O\in\calO_q:\ \alpha(O)\in \calC_0\}\Big|\cdot f_q^{\calC_0}(n)=\left(\sum_{\calA \in S_0} \sum_{T \in \calT_{\calA,0}} \frac{1}{c_T}\psi_\calA(q)\right)f_q^{\calC_0}(n).\]

We conclude by Proposition \ref{prop:sequence} that $f_q(n)$ is a two-level polynomial of degree function $g(q)$ and depth function
\[\min_r\big(e_r(q)+r\big)=\min_{r:\ e(q)\le r-1}\left(-1+r\right)= -1+\big(e(q)+1\big)=e(q).\]
\end{proof}

\section{Graph polynomials} \label{sect:chromatic}
Given a graph $G$, the chromatic polynomial $\chi_G(n)$ is the number of ways to color the vertices of $G$ with $n$ colors such that two adjacent vertices receive different colors. This function is well known to be a polynomial and its degree is the number of vertices of $G$. For several interesting infinite classes of graphs $\{G_q\}_{q\in\NN}$ with high symmetry, we will see that $\chi_{G_q}(n)$ is a two-level polynomial of infinite depth. With similar methods, we show that the characteristic polynomials and matching polynomials of the same classes of graphs are also two-level polynomials of infinite depth. 

\subsection{Two intermediate examples}
Before tackling some more complex families of graphs, we examine two families for which we recover complete formulas. The first one will reinforce our schema (Theorem~\ref{thm:schema}), and the second will introduce inclusion-exclusion as an important tool in later examples.

\begin{example}
Let $G_{q}=K_{q,q}$, the complete bipartite graph with two parts $A_q$ and $B_q$ of size $q$, and let $f_q(n)=\chi_{G_q}(n)$. We will show that $f$ is a two-level polynomial of degree function $2q$ and infinite depth. We use the numbering from the schema.
\begin{enumerate}
\item[(S1)] A proper coloring induces a partition $\Pi$ of $A_q$, where each part is the set of vertices of a given color. If $\Pi$ has $i$ parts, then the number of colorings of $A_q$ that induce $\Pi$ is $(n)_i=n(n-1)\cdots(n-i+1)$. The vertices of $B_q$ can each be assigned any of the remaining $(n-i)$ colors. Therefore the total number of proper colorings of $G_q$ that induce the partition $\Pi$ on $A_q$ is $(n)_i(n-i)^q$, and 
\[\chi_{G_q}(n) = \sum_{i}\sum_{\Pi:\ \abs{\Pi}=i}(n)_i(n-i)^q.\]
Our objects will be the partitions $\Pi$ of $A_q$.

\item[(S2)] The crux corresponding to $\Pi$ will be the collection of all parts that have size at least two. The crux isomorphisms are the obvious ones so that isomorphism classes will be indexed by multisets of part sizes (each size at least 2). If a crux contains $y$ elements of $[q]$ in $z$ parts, then $\Pi$ has $q-y$ singleton parts, and $\abs{\Pi} = z+q-y$.
Define $r(\Pi)=y-z$, so that $r(\Pi)=q-\abs{\Pi}$. The the crux, together with $q$, determines the contribution
\[(n)_{q-r(\Pi)}\big( n-q+r(\Pi) \big)^q \]
to the sum. This is a two-level polynomial of infinite depth by Example~\ref{ex:rising_fact}, Proposition~\ref{prop:products}, and Proposition~\ref{prop:powers}.

\item[(S3)] The partition into singletons yields the unique term, $(n)_q(n-q)^q$, of maximal degree $2q$, and the codegree of a given term is therefore $r(\Pi)$.

\item[(S4)] The size of the crux is $y$, the number of elements of $[q]$ it contains. Since each part in the crux has size at least two,
\[r(\Pi) = y-z \ge y-\frac{y}{2}=\frac{y}{2},\]
and so the size of a crux is bounded by $2r(\Pi)$.

\item[(S5)] The foundation is simply the set of elements of $[q]$ that belong to the crux. Its size is simply its cardinality, which is also the size of any associated crux. As in Example~\ref{ex:rising_fact}, important information about the crux (the number of parts in the crux) is lost when passing to foundations, and so different cruxes corresponding to the same foundation may contribute different terms to the sum.

\item[(S6)] This time, we define the foundation isomorphisms to be any bijection (not necessarily order-preserving). There is exactly one foundation isomorphism class of a given size.

\item[(S7)] Isomorphisms of foundations induce isomorphisms of cruxes in the obvious way, and this time they induce \emph{all} crux isomorphisms. That is, the crux types are exactly the crux isomorphism classes, rather than a proper refinement. (We are doing Example~\ref{ex:rising_fact} and this one in slightly different ways in order to highlight some key counting principles.)

\item[(S8)] Any crux type mapping to a foundation of size $a$ can be realized by a partition of $[a]$ with all parts of size at least two. There are only finitely many such partitions. 

\item[(S9)] The number of isomorphisms with domain $[a]$ and codomain a foundation in $[q]$ is $(q)_a$. This suffices to apply the schema and conclude we have a two-level polynomial of infinite depth. Additionally, we can calculate the constant $c_T$ in Lemma~\ref{lem:automorph} to get an explicit formula, as follows.

 A crux type is indexed by a multiset $T=[t_1^{e_1},\ldots, t_k^{e_k}]$ where $t_j\ge 2$ are the sizes of the parts and $\sum_j e_jt_j$ is the size, $a$. The number of crux automorphisms (induced by foundation automorphisms) is
\[c_T=\prod_j (t_j!)^{e_j}e_j!.\]
\end{enumerate}
Putting it all together, let $\calT_a$ be the set of partitions of $[a]$ with each part of size at least two, and for $T \in \calT_a$, let $c_T$ be defined as above and $r(T)=a-\sum_j e_j$ the codegree of the term. Then
\[\chi_{G_q}(n) = \sum_{a}\sum_{T\in \calT_a} \frac{(q)_a}{c_T}\cdot (n)_{q-r(T)}\cdot\big(n-q+r(T)\big)^q.\]

Swenson \cite{Swenson} obtained a formula for the chromatic polynomial of $K_{q,q}$ using Stirling numbers of the second kind (the number $S(q,j)$ of partitions of $[q]$ into $j$ parts). We note that the key calculation hidden in our proof that gives us a two-level polynomial is that $S(q,q-r)$ is a polynomial in $q$, for fixed $r$ (see, for example, Howard \cite{Howard}).
\end{example}

The next example illustrates the use of inclusion-exclusion in the schema. It also shows how we can leverage the results of Section \ref{sect:alg} to create more and more complicated two-level polynomials.

\begin{example}
  Define the Cartesian product $G\times H$ of two graphs $G, H$ to be the graph with vertices $(u,v)$ such that $u$ is a vertex of $G$ and $v$ a vertex of $H$, and edges $(u,v)\!\sim\!(u,v')$ for each edge $v\!\sim\!v'$ of $H$ and each vertex $u$ of $G$, as well as $(u,v)\!\sim\!(u',v)$ for each edge $u\!\sim\!u'$ of $G$ and each vertex $v$ of $H$. Now let $G_q =  K_q\times P_q$, where $K_q$ is the complete graph on $q$ vertices and $P_q$ is the path on $q$ vertices. Let $f_q(n)=\chi_{G_q}(n)$ be the chromatic polynomial, which we will show is a two-level polynomial of infinite depth.

We think of $K_q\times P_q$ as $q$ copies of $K_q$, with vertices in the $i$-th copy connected to their corresponding vertices in the $(i+1)$-st, for $1\le i\le q-1$. The first copy of $K_q$ may be properly colored in $(n)_q$ ways, which we proved in Example~\ref{ex:rising_fact} to be a two-level polynomial of infinite depth. Given a proper coloring of $K_q$, let $g_q(n)$ be the number of ways to extend this to a proper coloring of $K_q\times P_2$ by coloring an adjacent $K_q$. If we can show that $g_q(n)$ is a two-level polynomial of infinite depth, then so is
\[f_q(n)=(n)_q\big(g_q(n)\big)^{q-1}\]
by Propositions \ref{prop:products} and \ref{prop:powers}.

Given $A\subseteq [q]$ with $\abs{A}=a$, the number of proper colorings of the second copy of $K_q$ where the vertices in $A$ get the \emph{same} color in the two copies is $(n-a)_{q-a}$: the vertices in $A$ have their color determined, and the remaining $q-a$ vertices can be colored with the remaining $n-a$ colors. We want colorings with no such conflicts, so inclusion-exclusion yields
\[g_q(n) = \sum_{A\subseteq [q]} (-1)^{\abs{A}} (n-\abs{A})_{q-\abs{A}}.\]
In this case, the \emph{types} are easy and we don't really need the schema: all subsets $A$ of the same cardinality can be of the same type, which gives us
\[g_q(n)=\sum_{a=0}^{\infty} (-1)^{a} \binom{q}{a} (n-a)_{q-a},\]
a two-level polynomial of infinite depth.
\end{example}

\subsection{Chromatic polynomials via M\"obius inversion}
For more complicated families of graphs, we will calculate each chromatic polynomial via the technique of M\"obius inversion on the bond lattice, as developed by Rota \cite[\S 9]{Rota}. Given a partition of the vertices of a graph $G = (V, E)$ such that the induced subgraph of each part is connected, define its \emph{bond} to be the disjoint union of these induced subgraphs. The bonds form a ranked lattice $\calL_G$ under the partial order given by coarsening the partitions. The rank of a bond is thus the number of vertices minus the number of connected components. 

Any coloring of the vertices of $G$ induces a bond whose edges are the edges of $G$ that are not properly colored (that is, both ends of the edge are assigned the same color.) The proper colorings are those for which the bond is the finest partition into singletons. Given a bond $B$ on $y$ vertices of rank $r$ in the lattice, the number of $n$-colorings whose bond equals or coarsens $B$ is $n^{|V|-r}$. It follows by M\"obius inversion that the chromatic polynomial is 
\begin{equation} \chi_G(n) = \sum_{B \in \calL(G)}\mu_{\calL(G)}([\hat{0},B])n^{|V|-\textup{rank}(B)} =: \sum_{B \in \calL(G)} p^{B}(n). \label{eq:Moebiusonbonds} \end{equation}
so we will apply the schema by letting objects be the bonds. We will define the crux $\calC$ of a bond $B$ to be the union of all of the components of size greater than one. We take the size of a crux to be the number of vertices and crux isomorphisms to be graph isomorphisms between cruxes. 

\begin{proposition} \label{prop:allbonds}
Let $G_q$ be any family of graphs such that the number of vertices of $G_q$ agrees with a numerical polynomial $g(q)$. Then the family of chromatic polynomials $\chi_{G_q}(n)$ satisfies hypotheses (S1), (S2), (S3), and (S4) in Theorem \ref{thm:schema}.
\end{proposition}

\begin{proof}
  We already have (S1) by (\ref{eq:Moebiusonbonds}). An isomorphism class of cruxes is an unlabeled graph $\calC$, say on $y$ vertices and $z$ connected components, with no isolated vertices. If $B$ is a bond of $G_q$ whose crux is $\calC$, then it has $g(q)-y+z$ connected components and its rank is $y-z$, independent of $q$. Also, the intervals $[ \hat{0}, B]$ in $\cL_q := \calL_{G_q}$ and $[ \hat{0}, \calC]$ in $\cL_\calC$ are isomorphic, so their M\"obius function values are equal. It follows that
  \[ p^{B}(n) = \mu_{\calL_q}\big([ \hat{0}, B]\big)n^{g(q)-y+z} = \mu_{\calL_\calC}\big([ \hat{0}, \calC]\big)n^{g(q)-y+z} =: f_q^{\calC}(n), \]
  which is a two-level polynomial of degree function $n^{g(q)-y+z}$ and infinite depth, by Corollary \ref{cor:powers}. 
  This proves (S2). The codegree $r(\calC) = y-z$ is a nonnegative constant and is equal to zero if and only if $\calC$ is the graph $\calC_0$ with no vertices, proving (S3). Finally, since each connected component of $\calC$ has at least two vertices,
  \[ r(\calC) = y-z \geq y - \frac{y}{2} = \frac{y}{2} \]
  which yields the bound $y \leq 2r(\calC)$, proving (S4). 
\end{proof}

The remaining parts of the schema depend on how the foundation of a bond is defined, which will vary depending on the structure of the graphs.

\subsubsection{Kneser and Johnson graphs}\label{subsect:Kneser}

We next illustrate the method as applied to families of \emph{Kneser graphs} \cite{Kneser}. Fix $k \geq 0$. For each $q \geq 0$, the Kneser graph $\Kn(q,k)$ is the graph on the $k$-subsets of the symbol set $[q]$ whose edge relation is disjointness. The first nontrivial case is $\Kn(5,2)$ which is the Petersen graph. Kneser conjectured that the chromatic number of $\Kn(q,k)$ is $q-2k+2$ whenever $q \geq 2k+1$. This conjecture was first proved by Lov\'asz \cite{Lovasz} using topological methods. We do not know of any formula for the chromatic \emph{polynomial} of $\Kn(q,k)$.  

\begin{theorem} \label{thm:Kneser}
 Fix $k \geq 1$. Then $f_q(n) = \chi_{\Kn(q,k)}(n)$ is a two-level polynomial of degree function $\binom{q}{k}$ and infinite depth.
\end{theorem}

\begin{proof} Let $\cL_q$ be the bond lattice of the Kneser graph $\Kn(q,k)$ and define cruxes as above, so (S1), (S2), (S3), and (S4) of the schema follow from Proposition \ref{prop:allbonds}. We define the foundation $A$ of a bond $B \in \cL_q$ to be the set of symbols that belong to at least one vertex in the crux of $B$. We define the size of the foundation to be its cardinality and foundation isomorphisms to be bijections, so (S6) is immediate: there is only one foundation of each size up to isomorphism. Since there are $k$ symbols at each vertex, the foundation size is at most $k$ times the crux size (with equality when all the vertices in the crux represent disjoint subsets), proving (S5). A foundation isomorphism $\sigma: A\rightarrow A'$ induces a function on the vertices of the crux, which are $k$-subsets of $A$, via
  \[ \overline{\sigma}(v) = \{ \sigma(i): i \in v \}.\]
  This map respects composition and inverse and also preserves the edge relation of disjointness of symbol sets, proving (S7).

If $T$ is a type of bond of foundation size $a$, then up to foundation isomorphism $T$ can be realized by a bond whose crux is just the fixed set $[a]$, which is a graph on $[a]$ with some additional properties. There are only finitely many such graphs, proving (S8). Finally, the number of foundation isomorphisms from $[a]$ to a subset of $[q]$ is $(q)_a$, which is a polynomial for all $q\in\NN$ (even for $q<a$, for which it evaluates to zero). Therefore we have a two-level polynomial of infinite depth by Theorem \ref{thm:schema}.
\end{proof}

\begin{remark} For Kneser graphs, the map from foundation isomorphisms to crux isomorphisms is neither injective nor surjective in general. 

  \begin{itemize}
  \item For $q \geq 4$, let $\sigma$ be the induced subgraph of $\Kn(q,2)$ on the vertices $12$ and $34$, which is a single edge. Then the nontrivial foundation automorphism that transposes 1 with 2 induces the trivial (identity) crux automorphism.
    
  \item For $q \geq 8$, let $\sigma$ be the induced subgraph of $\Kn(q,3)$ on the vertices 123, 456, and 178 and $\sigma'$ be the induced subgraph on 123, 456, and 127. Then $\sigma$ and $\sigma'$ are both paths of length two, so they are isomorphic as cruxes. But the foundations are of different sizes, so no foundation isomorphism induces this crux isomorphism.   
  \end{itemize}
 The second example indicates why it would be difficult to enumerate a crux isomorphism class directly and we need to use foundations. 
\end{remark}

A related family of graphs are the \emph{Johnson graphs} $J(q,k)$, whose vertices are again the $k$-subsets of $[q]$, but two vertices are adjacent exactly when they share $k-1$ elements. Johnson graphs arise in coding theory \cite{EB} and are central to Babai's recent proof \cite{Bab} that the graph isomorphism problem can be solved in time $\exp\left( (\log n)^{O(1)}\right)$.\footnote{This time function is called quasipolynomial in Babai's paper, but this is not related to our notion of quasi-polynomial.} Their chromatic numbers are not known in general \cite{GoMe}.

\begin{corollary} \label{cor:Johnson}
 Fix $k \geq 1$. Then $f_q(n) = \chi_{J(q,k)}(n)$ is a two-level polynomial of degree function $\binom{q}{k}$ and infinite depth.
\end{corollary}

\begin{proof}
We define the foundation, foundation size, and foundation isomorphisms just as for Kneser graphs, and the proofs of (S5) -- (S9) are just as in Theorem \ref{thm:Kneser}. 
\end{proof}

\subsubsection{Products of Complete Graphs}

\begin{theorem} \label{thm:completegraphpowers}
For fixed $k$ and each value of $q$, let $G_q = K_q^k$ be the Cartesian product of $k$ copies of $K_q$. Then $\chi_{G_q}(n)$ is a two-level polynomial of infinite depth. 
\end{theorem}

\begin{proof}
  We take objects to be the bonds as before, so (S1), (S2), (S3), and (S4) hold by Proposition \ref{prop:allbonds}. A vertex of $G_q$ is a tuple $(i_1, \dots, i_k)$ where $i_j \in [q]$ for each $j$, and an edge is a pair of tuples that differ in exactly one index. Define the foundation of a bond $B$ to be the $k$-tuple of sets $\calA = (A_1, \dots, A_k)$ where
  \[ A_j = \{ i \in [q]: \; \exists (i_1, \dots, i_k) \in \textup{Vert}(\textup{crux}(B)) \textup{ such that } i_j = i\}.\]
  Let $a_j = |A_j|$ for each $j$ and define the size of the foundation $\calA$ to be $\sum_{j=1}^k a_j$. For each $j$, the distinct elements of $A_j$ are necessarily obtained from distinct vertices of the crux, so
  \[ \textup{size}(\calA) = \sum_{j=1}^k a_j \leq k \max\{a_1, \dots, a_k\} \leq k \textup{ size}(\calC),\]
  proving (S5).

  We define a foundation isomorphism from $(A_1, \dots, A_k)$ to $(B_1, \dots, B_k)$ to be any $k$-tuple $(\phi_1, \dots, \phi_k)$ where $\phi_j$ is a bijection from $A_j$ to $B_j$. An isomorphism class of foundations of size $u$ is a $k$-tuple of set sizes whose sum is $u$, so these classes are in bijection with $k$-compositions of $u$, proving (S6). Any permutation of the vertices within each copy of $K_q$ in $G_q$ is a graph automorphism. Since the foundation (unlike the case of Kneser graphs) is simply a subset of the vertices, there is no issue with composition or inverses of these automorphisms, proving (S7). Given a crux isomorphism class $\calC$, up to foundation isomorphism we may choose an object whose crux class is $\calC$ and whose foundation is a tuple of the form $([a_1], \dots, [a_k])$. The vertices of the crux must then be a subset of the fixed set $A_1 \times\cdots\times A_k$. There are only finitely many graphs on a fixed set, proving (S8). Finally, the number of foundation isomorphisms with domain $([a_1], \dots, [a_k])$ is $(q)_{a_1} (q)_{a_2} \cdots (q)_{a_k}$ which is a polynomial for all $q \in \NN$, proving (S9).   
\end{proof}

\begin{remark} We can generate more families of graphs whose chromatic polynomials are two-level polynomials via Cartesian products with fixed graphs. For example, let $H$ be a fixed finite graph. For fixed $k$, the functions
  $\chi_{(K_q)^k \times H}(n)$, $\chi_{\Kn(q,k) \times H}(n)$, and $\chi_{J(q,k) \times H}(n)$ are two-level polynomials of infinite depth. We may further combine graphs like $K_q$ and $\Kn(q,k)$ freely, even using other graph products (tensor product, strong product) that have a similar localized structure. The key is that the graphs have sufficiently large automorphism groups so that foundations can simply be sets of some fixed size. We leave the details to the reader.
\end{remark}

\subsection{Other graph polynomials}
\begin{remark} \label{rem:notonlybonds}
Many of the ideas used to prove that families of chromatic polynomials yield two-level polynomials extend naturally to other polynomials associated to graphs. In particular, the proofs of (S5) through (S9) in Theorems \ref{thm:Kneser} and \ref{thm:completegraphpowers} depend on the definition of foundation and foundation isomorphism but not on the definition of cruxes as bonds. They only require that cruxes are a family of subgraphs closed under graph isomorphism. We will briefly cover two graph polynomials for which our approach extends easily, and one for which the question remains open. 
\end{remark}

\subsubsection{Characteristic polynomial}
Given a graph $G$, the characteristic polynomial is defined to be $p_G(n)= \det(nI - A)$, where $A$ is the adjacency matrix of $G$. By \cite[Proposition 1]{EM2}, we have 
\begin{equation} p_G(n) = \sum_{S\subseteq V(G)} (-1)^{\abs{S}}a_S \cdot n^{d-\abs{S}}, \label{eqn:characteristic} \end{equation}
where $S$ is a subset of the vertices of $G$, $a_S$ is the determinant of the adjacency matrix of the induced graph with vertex set $S$, and $d$ is the number of vertices of $G$.

\begin{corollary} \label{cor:characteristic}
  For fixed $k$, the characteristic polynomials of the familes $\Kn(q,k)$, $J(q,k)$, and $K_q^k$ are two-level polynomials of infinite depth. 
\end{corollary}

\begin{proof} By Remark \ref{rem:notonlybonds}, it suffices to prove the first four properties of the schema, and the proofs of these four will be identical for all three families.  
\begin{enumerate}
\item[(S1)] We have written the sum in (\ref{eqn:characteristic}). The objects are the sets $S\subseteq V(G_q)$. Given $S$, the term $(-1)^{\abs{S}}a_S \cdot n^{d(q)-\abs{S}}$ is polynomial in $n$.
\item[(S2)] Define the crux of an object $S$ to be $S$ itself (or equivalently, the induced subgraph on $S$). If $S$ is of type $\mathcal C$, then $f^{\mathcal C}_q(n)=(-1)^{\abs{S}}a_S \cdot n^{d(q)-\abs{S}}$ is a two-level polynomial of infinite depth.
\item[(S3)] The degree of the characteristic polynomial is $d(q)$, obtained when $S$ is the empty set. The codegree of $f^{\mathcal C}_q(n)$ is therefore $r(\mathcal C)=\abs{S}$.
\item[(S4)] The size of the crux is simply $\abs{S}$ which is exactly the codegree.
\end{enumerate}
\end{proof}

\subsubsection{Matching polynomials}
Matching polynomials (also called acyclic polynomials) were studied in the context of theoretical chemistry and then for their general mathematical properties; matching polynomials of specific families of graphs yield important families such as Chebyshev, Hermite, and Laguerre polynomials \cite{GG}. 

Let $G$ be a graph with $d$ vertices, and let $\mathcal{M}_i(G)$ be the set of $i$-matchings on $G$: that is, collections of $i$ edges involving $2i$ distinct vertices. Then define the matching polynomial (see \cite[Section 3.1]{EM2}) to be
\begin{equation} \mu_G(n)=\sum_{i\ge 0} (-1)^i \abs{\mathcal{M}_i(G)}n^{d-2i}. \label{eqn:matchingpolyn} \end{equation}
\begin{corollary} \label{cor:matching}
  For fixed $k$, the matching polynomials of the familes $\Kn(q,k)$, $J(q,k)$, and $K_q^k$ are two-level polynomials of infinite depth. 
\end{corollary}
  Again, it suffices to prove the first four properties of the schema. Taking the objects to be $i$-matchings, we can write (\ref{eqn:matchingpolyn}) as 
  \[\mu_{G_q}(n)=\sum_{i\ge 0} (-1)^i \sum_{o\in \mathcal{M}_i(G_q)}n^{d-2i},\]
  establishing (S1). Let the crux be the induced subgraph on the set of matched vertices, which will have $2i$ vertices. Then (S2) through (S4) follow as in Corollary \ref{cor:characteristic}, noting that the size of the crux ($2i$) is exactly the codegree.
  
\subsubsection{An open question: flow polynomials}
The flow polynomial \cite[\S IX.4]{Tutte01} of a graph $G = (V,E)$, which counts nowhere zero flows modulo $n$, can be defined as
\[ F_G(n) = (-1)^{|E|} \sum_{S \subseteq E} (-1)^{|S|} n^{\rho_1(G;S)} \]
where $\rho_1(G;S)$ denotes the rank of the first homology group of the graph $(V,S)$. Like the chromatic, characteristic, and matching polynomials, it can be obtained as an evaluation of the Tutte polynomial \cite{EM1}. As with the characteristic polynomial, one might apply the schema by taking the objects to be the subsets $S \subseteq E$ to obtain (S1). Again, the crux would be the subgraph $(V,S)$ to obtain (S2). A highest degree term occurs when $S = E$, so it would be natural to define the ``size'' of a crux to be the number of edges in the complement of $S$. If the graph is bridgeless, then this term is unique, obtaining (S3). The assumption of bridgelessness is harmless because the flow polynomial is identically zero if $G$ has a bridge. However, (S4) will not generally hold: we can often remove many edges without affecting the first Betti number.

For example, for the family of $q$-cycles there is a single term of degree one given by $S = E$ and all of the other terms ($2^{q} - 1$ of them) are of degree zero, so we cannot directly apply the schema in this way. However, the flow polynomial of the $q$-cycle is $n-1$, which \emph{is} a two-level polynomial.

\begin{question} For fixed $k$, do the flow polynomials of $\Kn(q,k)$, $J(q,k)$, and $K_q^k$ form two-level polynomials? 
\end{question}

\section{Ehrhart theory, chess queens, and Sidon sets} \label{sect:Ehrhart}
An equivalent way to define the bond lattice of a graph $G$ is as the lattice of flats of the arrangement of hyperplanes $x_i = x_j$ for each edge $ij$ of $G$. From this point of view, the proper colorings of $G$ are in bijection with the lattice points in the interior of the dilated $q$-cube $(n+1)C_q^\circ$ (where $C_q=[0,1]^q$) that do not lie on any of the hyperplanes, and each term $n^{v-\textup{rank}(B)}$ in (\ref{eq:Moebiusonbonds}) counts the number of lattice points in the intersection of the flat corresponding to $B$ with $(n+1)C_q^\circ$. Ehrhart theory (see \cite{BR}) guarantees that the number of lattice points will be quasi-polynomial in $n$, and the \emph{inside-out polytopes} of Beck and Zaslavsky \cite{BeZa} make the M\"obius inversion calculations precise. In particular, for families of counting problems that can be expressed using highly symmetric arrangements, one can often show that the counting functions are two-level polynomials. This is in fact how the results on chess queens (and other pieces that move arbitrarily far in straight lines) \cite{CHZ} that motivated our work were proved. To illustrate the generality of this approach, we will use it first to reprove the queens result and then to prove an analogous result for the apparently unrelated problem of counting Sidon sets of length $n$ and size $q$. As another illustration, we extend the queens result to cover the case of mutually nonattacking black and white queens, which requires using a subspace arrangement rather than a hyperplane arrangement.

Let $D=\RR^d$, $P$ be a full-dimensional rational polytope in $D$, and $b \geq 1$. We will be interested in lattice points in $\left(nP^q \right)^\circ$, the interior of the $n$th dilate of the Cartesian product of $q$ copies of $P$. Let $\calH$ be a linear hyperplane arrangement (that is, a finite set of hyperplanes that all contain the origin)  in $D^b=\RR^{db}$.  For each $q$, define a hyperplane arrangement $\calH_q$ in $D^q$ as follows: for each hyperplane $H \in \calH$ and each injection $\sigma$ from $[b]$ to $[q]$, there is a hyperplane $H_{\sigma, q} \in \calH_q$ obtained by reindexing copies of $D$ according to $\sigma$. It follows that $\calH_q$ is invariant under permutations of the $q$ copies of $D$.

\begin{theorem} \label{thm:allEhrhart}
   Under the hypotheses of the preceding paragraph, let $f_q(n)$ denote the number of lattice points in $\left(nP^q \right)^\circ$ that avoid all of the hyperplanes in $\calH_q$. Then $\frac{1}{\Vol(P)^q}f_q(n)$ is a two-level polynomial of infinite depth. 
  \end{theorem}

\begin{proof}
Let $\calL_q$ be the lattice of flats of $\calH_q$ (the nonempty intersections of hyperplanes in $\calH_q$, ordered by reverse inclusion), and $\overline{\calL_q}$ be the sublattice of $\calL_q$ consisting of the flats that intersect the interior of $P^q$. Both of these lattices are again invariant under permutation of the $q$ copies of $D$. To apply the schema, we define our objects to be the flats in $\overline{\calL_q}$.  By M\"obius inversion, as in the paper on inside-out polytopes \cite{BeZa}, we obtain
\begin{equation} f_q(n) = \sum_{F \in \overline{\calL_q}} \mu\big([\hat{0},F]\big)E_{F}(q;n) \label{eqn:inside-out} \end{equation}
  where $E_{F}(q;n)$ is the number of lattice points in $\left(nP^q\right)^\circ \cap F$, which is a quasi-polynomial in $n$ by Ehrhart's theorem. For each flat $F \in \overline{\calL_q}$, there is a unique subset $A \subseteq [q]$ such that
  \[ F = F' \times D^{[q] \setminus A} \]
  where $F' \subseteq D^A$ is a subspace satisfying the property that for each $i \in A$, at least one of the coordinates on the copy of $D$ indexed by $i$ is restricted by some hyperplane in $\calH_q$ that contains $F$. We define the crux of $F$ to be the subspace $F'$ and the foundation of $F$ to be the set $A$. We define foundation isomorphisms to be set bijections and the size of a foundation to be its cardinality. In addition, we simply define the size of the crux to be the size of its foundation and a crux isomorphism to be any vector space isomorphism induced by relabeling the copies of $D$ according to a foundation isomorphism. 

  We proceed to verify the nine hypotheses of the schema, with depth function $e(q) = \infty$ in (S9) for every type.  We have already seen that (S1) holds by Ehrhart's theorem and (\ref{eqn:inside-out}). Given a flat $F \in \overline{\calL_q}$ of crux $F'$ and foundation $A$ of size $a$, we have by the definition of crux that 
    \[ P^q \cap F = (P^A \cap F') \times P^{[q] \setminus A} \]
    which, defining $E_S(n)$ to be the number of lattice points in $nS$ (for a set $S$), implies that 
    \[ E_F(q;n) = E_{P^A \cap F'}(n) \cdot E_P(n)^{q-a}.\]
    
    Since the hyperplane arrangement is linear, both $E_{P^A \cap F'}(n)$ and $E_P(n)$ are quasi-polynomials in $n$ by Ehrhart's theorem, and the leading coefficient of $E_P(n)$ is $\Vol(P)$. By Corollary \ref{cor:powers}, $E_P(n)^{q-a}/\Vol(P)^{q-a}$ is a two-level quasi-polynomial of infinite depth. Dividing by the constant $\Vol(P)^a$ and multiplying by the quasipolynomial $E_{P^A \cap F'}(n)$, we conclude by Proposition \ref{prop:products} that $E_F(q;n)$ is again a two-level quasi-polynomial of infinite depth. Furthermore, a crux isomorphism does not affect either factor: $E_{P^A \cap F'}(n)$ is preserved because of the symmetry of $\calH_q$, and $E_P(n)^{q-a}$ does not depend on the crux at all. This proves (S2). 
    
    Also, $P^q$ is full-dimensional in $D^q = \RR^{dq}$ and so the largest flat in $\overline{\calH_q}$, obtained as the intersection of $(P^q)^\circ$ with no hyperplanes, indexes a term of degree exactly $dq$. Let $F$ be any flat of codimension $r > 0$ in $\overline{\calL_q}$. Since $F$ intersects $\left(P^q\right)^\circ$, the term $E_F(q;n)$ is of degree $dq-r$, proving (S3). Furthermore,  we can write $F$ as the intersection of $r$ hyperplanes in $\calH_q$. Since each hyperplane $H \in \calH_q$ is obtained by reindexing a hyperplane in the fixed arrangement $\calH$, there is a maximum number $M$ of coordinates that can be restricted in $H$ independent of $q$. Thus at most $Mr$ coordinates are restricted in $F$. In particular, the crux size is also bounded by $Mr$, proving (S4).



Next, (S5) is immediate because crux size and foundation size coincide, and (S6) holds because there is just one set of size $a$ up to set bijection. Since crux isomorphisms are defined to be the linear maps induced by permuting copies of $D$ via foundation isomorphisms, (S7) is also immediate. Given a number $a$, if $T$ is a type of flat of foundation size $a$ then (via set bijections) there is a flat of type $T$ in the fixed lattice $\overline{\calL_a}$, so there can be only finitely many such types, proving (S8). Finally, if $|A| = a$ then the number of foundation isomorphisms from $A$ to a subset of $[q]$ is $(q)_a$, which is a polynomial for all $q$ and vanishes for $q > a$, proving (S9) with $e(q) = \infty$ regardless of the type. 
\end{proof}

We illustrate how our method easily reproves the motivating result of Chaiken, Hanusa, and Zaslavsky. For simplicity, we prove only the special case of nonattacking queens on an ordinary $n \times n$ board, but as in their theorem, the same argument works for any piece whose allowed moves are the integer multiples of a fixed, finite set of integer directions, and on any lattice polygonal board. 

\begin{corollary}\cite[Theorem 4.2]{CHZ} \label{cor:queens} The number of placements of $q$ labeled nonattacking queens on an $n \times n$ chessboard is a two-level quasi-polynomial of infinite depth. 
\end{corollary}

\begin{proof}
We take $d=2$ so $P$ is the unit square, with coordinates $(x_i, y_i)$ on each copy of $D = \RR^2$. Let $f_q(n)$ be the number of lattice points in $\left( nP^q \right)^\circ$ that avoid the hyperplanes $x_i = x_j$, $y_i = y_j$, $x_i+y_i = x_j+y_j$ and $x_i - y_i = x_j-y_j$ whenever $1 \leq i < j \leq n$. All of these hyperplanes are linear, and they are obtained in the manner prescribed by Theorem \ref{thm:allEhrhart} from the arrangement in $D^2$ given by 
\[ \calH = \{ x_1 = x_2; \; y_1 = y_2; \; x_1 + y_1 = x_2 + y_2; \; x_1 - y_1 = x_2 - y_2\}. \]
By Theorem \ref{thm:allEhrhart}, $f_q(n)$ is a two-level quasi-polynomial of infinite depth. Now since $f_q(n)$ counts points in the interior, the coordinates of each queen are integers between $1$ and $n-1$, so $f_q(n)$ represents placements on an $(n-1) \times (n-1)$ board. But by Proposition \ref{prop:evaluation}, $f_q(n+1)$ is again a two-level quasi-polynomial of infinite depth, completing the proof. 
\end{proof}

\subsection{Sidon sets} A \emph{Sidon set} \cite{Sidon} (also known as a Golomb ruler) of length $n$ and size $q$ is a sequence of numbers $0 = s_1 < s_2 < \dots < s_q = n$ such that all of the pairwise differences are distinct. The most-studied question about such sets is how dense they can be; that is, how large can $q$ be as a function of $n$? See \cite{O'B} for results on this and many related problems. Beck, Bogart, and Pham \cite{BBP} studied the somewhat different problem of counting Sidon sets when both $n$ and $q$ are given. They showed using inside-out polytopes that for each $q$, the counting function is a quasi-polynomial in $n$.

To obtain a two-level quasi-polynomial, we slightly modify the problem. First, we consider sets of length \emph{at most} $n$; that is, we do not require $s_1 = 0$ nor $s_q = n$. The original problem is recovered via successive differences. Second, we consider \emph{ordered Sidon sets}: that is, tuples $(s_1, \dots s_q)$ without requiring that the elements be in increasing order. Since the elements of a Sidon set are always distinct, this simply multiplies the counting function by $q!$. We also translate the sets to lie between 1 and $n+1$ rather than between 0 and $n$, which does not affect the count. 

\begin{theorem} \label{thm:Sidon}  
  The number of ordered Sidon sets of length at most $n$ and size $q$ is a two-level quasi-polynomial of degree function $q$ and infinite depth. 
\end{theorem}

\begin{proof} We take $d=1$ so $D = \RR$ and $P$ is just the unit interval. For given $q$ and $n$, the modified counting problem reduces to counting lattice points in $\left((n+2)P^q\right)^\circ$ that avoid the arrangement $\calH_q$ of hyperplanes of the forms $x_i = x_j$, $x_i + x_j = 2x_k$, and $x_i - x_j = x_k - x_\ell$ for distinct indices $i,j,k,\ell$. All of these hyperplanes are linear, and they are obtained as prescribed in Theorem \ref{thm:allEhrhart} from the arrangement in $D^4$ given by 
    \[ \calH = \{ x_1 = x_2; \; x_1 + x_3 = 2x_2; \; x_1 - x_4 = x_2 - x_3 \}. \]
As in the proof of Corollary \ref{cor:queens}, the result follows immediately from Theorem \ref{thm:allEhrhart} and Proposition \ref{prop:evaluation}.
\end{proof}

\begin{remark}
For fixed $q$, $f_q(n)$ is a quasi-polynomial, and therefore it can be evaluated at negative integers.
Beyond examining the above $f_q(n)$ for Sidon sets, Beck, Bogart, and Pham \cite{BBP} further analyze $(-1)^{g(q)}f_q(-n)$, which also has a combinatorial interpretation. In fact, a main aim in the study of inside-outside polytopes is to understand \emph{reciprocity}, that is, combinatorial interpretations of functions such as $(-1)^{g(q)}f_q(-n)$. We note that if $f_q(n)$ is a two level polynomial, then $(-1)^{g(q)}f_q(-n)$ is as well, by Proposition~\ref{prop:evaluation}.
\end{remark}

\begin{remark} In this application as well as the previous one, all of the hyperplanes in $\calH_q$ meet at the point $(1/2, \dots, 1/2)$ in the interior of $P^q$. This is not critical for our applications (though it does imply $\overline{\calL_q} = \calL_q$ in the proof of Theorem~\ref{thm:allEhrhart}), but it is important for understanding reciprocity because it yields transverse intersection in the sense of \cite{BeZa}.
\end{remark}

This technique can be adapted to many examples with linear constraints that have rich families of group actions that do not necessarily satisfy all of the hypotheses of Theorem \ref{thm:allEhrhart}. Before we see one such example, a non-example is worth pointing out.

\begin{example} The number of \emph{anti-magic} $q\times q$ squares with entries in $[n]$ (where row, column, and possibly diagonal sums are required to be distinct)  is a classic application of inside-out polytopes, but writing them as two-level polynomials seems impossible: even the leading coefficient is a hard-to-compute function, because it is the volume of the Birkhoff polytope in dimension $q$ \cite{BP}. The key difference here is that although for each $q$ the hyperplane arrangement $\calH_q$ is invariant under permutation, there is no single hyperplane arrangement $\calH$ that generates $\calH_q$ for every $q$ by reindexing copies of $D$, as required by Theorem \ref{thm:allEhrhart}. In particular, the number of variables involved in each hyperplane is not bounded by a constant $M$ independent of $q$, so (S4) does not hold. 
\end{example}

\subsection{Subspace arrangements and black and white queens}
To end this section, we apply the schema to prove a new generalization of Corollary \ref{cor:queens} that requires working with subspace arrangements and a somewhat weaker symmetry hypothesis.

\begin{theorem} \label{thm:whiteblackqueens}
  The number of placements of $q$ labeled white queens and $q$ labeled black queens on an $n \times n$ board so that no pair of queens of \emph{opposite} color attack each other is a two-level quasi-polynomial.
\end{theorem}

\begin{proof} Let $(x_i, y_i)$ and $(z_i, w_i)$ be the respective positions of the $i$-th white queen and the $i$-th black queen. Taking $P$ to be the unit cube in $\RR^4$, such placements are indexed by the number $f_q(n+1)$ of lattice points in $\left((n+1)P^q\right)^\circ$ that avoid the hyperplanes $x_i = z_j$, $y_i = w_j$, $x_i+y_i = z_j + w_j$, and $x_i - y_i = z_j - w_j$ for all pairs of indices $(i,j)$ (not necessarily distinct), and also the \emph{subspaces} of codimension two given by $(x_i, y_i) = (x_j, y_j)$ or by $(z_i, w_i) = (z_j, w_j)$ for all pairs of distinct indices $(i,j)$. These subspaces are required to rule out two white queens or two black queens coinciding. (In the original queens problem such subspaces were not necessary because queens that coincide also attack each other in every possible direction.)

As well as involving a subspace arrangement $\calH_q$ rather than a hyperplane arrangement, this problem fails to satisfy the hypothesis in Theorem \ref{thm:allEhrhart} that each $\calH_q$ is obtained from a fixed arrangement $\calH$ in a space $D^b$ via all relabelings of copies of $D$. However, it suffices to establish an appropriate two-colored analogue of this hypothesis. Let $W$ be a copy of $\RR^2$ indexed by the coordinates $(x,y)$ of a white queen, and $B$ be another copy of $\RR^2$ indexed by $(z,w)$ for a black queen. Consider the subspace arrangement in $W^2 \times B^2$ given by $\calH =$ 
\[  \{x_1 = z_1, \, y_1 = w_1, \, x_1+y_1 = z_1+w_1, \, x_1 - y_1 = z_1 - w_1, \, (x_1, y_1) = (x_2, y_2), \, (z_1, w_1) = (z_2, w_2)\}.\] 
Then $\calH_q$ is obtained from $\calH$ as follows: for each pair of set injections $\sigma_1, \sigma_2: \{1,2\} \to [q]$, and each hyperplane or codimension two subspace in $\calH$, apply $\sigma_1$ to reindex the copies of $W$ and $\sigma_2$ to reindex the copies of $B$. In particular, $\calH_q$ is invariant under the action of $S_q \times S_q$ on $W^q \times B^q$ that separately permutes the copies of $W$ and of $B$. 

To show that $f_q(n+1)$ is a two-level quasi-polynomial of infinite depth, by Proposition \ref{prop:evaluation} it suffices to show that $f_q(n)$ is, which we do via the schema. As in \cite[\S8]{BeZa}, we work with the semilattice $\calL_q$ of flats of the subspace arrangement $\calH_q$. As in Corollary \ref{cor:queens}, all of the elements (hyperplanes and codimension-two subspaces) of $\calH_q$ meet at a common point in $(P^q)^\circ$, so $\overline{\calL_q} = \calL_q$. Although this semilattice is not graded, we still have for each $q$ that 
\[ f_q(n) = \sum_{F \in \calL_q} \mu([\hat{0},F])E_F(q;n) \]
which is a quasi-polynomial by Ehrhart's theorem, proving (S1).

Given a flat $F \in \calL_q$, let $A_1$ and $A_2$ respectively be the sets of white queens and of black queens constrained in $F$. Then we can write
\[ F = F' \times W^{[q] \setminus A_1} \times B^{[q] \setminus A_2} \]
and define the foundation to be the pair of sets ($A_1, A_2$) and the crux to be the subspace $F' \subseteq W^{A_1} \times B^{A_2}$. Foundation maps are pairs of set injections $(\sigma_1, \sigma_2)$ where $\sigma_i$ maps $A_i$ to $[q]$, and crux isomorphisms are the vector space isomorphisms induced by relabelling copies of $W$ according to $\sigma_1$ and the copies of $B$ according to $\sigma_2$, where $(\sigma_1, \sigma_2)$ is a foundation isomorphisms. The size of the foundation and of the crux are both taken to be $\abs{A_1} + \abs{A_2}$.

Now (S2) and (S3) follow just as in Theorem \ref{thm:allEhrhart}, using the symmetry of $\calH_q$. For (S4), if $F$ is a flat contributing a term of codegree $r$, then $F$ has codimension $r$ in $W^q \times B^q$. This implies that it is the intersection of \emph{at most} $r$ subspaces in $\calH_q$, and each subspace involves two queens, so the crux size is at most $2r$, proving (S4). The properties (S5) and (S7) are immediate from the direct relationship between crux and foundation. For a given foundation size $u$, there are exactly $u+1$ isomorphism classes given by the size of $A_1$ (between 0 and $u$), proving (S6). For the same reason, if $T$ is a type of flat of foundation size $u$, then each of the sets $A_1$, $A_2$ is of size at most $u$. Thus via pairs of set bijections, there is a flat of type $T$ in the fixed lattice $\overline{\calL_u}$, so there can be only finitely many such types, proving (S8). If $\abs{A_i} = a_i$ for $i = 1,2$, then there are $(q)_{a_1}(q)_{a_2}$ foundation isomorphisms to subsets of $[q] \times [q]$, proving (S9) with infinite depth.  
\end{proof}

\begin{remark} Chaiken, Hanusa, and Zaslavsky \cite[\S 7.6]{CHZ} observe that subspace arrangements would also be necessary to analyze higher-dimensional analogues of the queens problem. Such analogues could be formulated and proved in a manner similar to Theorem \ref{thm:whiteblackqueens}, but we agree with them that the resulting formulas would be much more complex. 
\end{remark}

\section{Chromatic polynomials with finite depth}
\label{sect:finite}

In Section~\ref{sect:chromatic}, we examined families of graphs $G_q$ that had high degrees of symmetry. For example, $S_q$ acts on $\Kn(q,k)$, or $S_q^k$ acts on $K_q^k$. These symmetries led to two-level chromatic polynomials whose depth was \emph{infinite}. In particular, the number of foundation isomorphisms were given by the expressions $(q)_a$ or $(q)_{a_1} (q)_{a_2} \cdots (q)_{a_k}$, which are valid even when $q<a$ or $q<a_i$ for some $i$, respectively.

Now we examine the families $P_q^k$ (the Cartesian product of $k$ copies of the path with $q$ vertices), and $C_q^k$ (the Cartesian product of $k$ copies of the cycle with $q$ vertices) with many fewer symmetries. This requires foundations that are more complicated, and the count (S9) of the number of foundation isomorphisms will be invalid for small $q$, leading to finite depth in the two-level polynomials.

As in Section~\ref{sect:chromatic}, we define objects to be bonds and define cruxes to be the union of all of the components of size greater than one, which immediately gives us Equation~(\ref{eq:Moebiusonbonds}), Proposition~\ref{prop:allbonds}, and (S1) through (S4) of the schema. For the foundations, we will use the \emph{projection tuple} of the crux $C$, which we define now.

\begin{definition}
\label{def:proj}
Given a Cartesian product $G=G_1\times\cdots\times G_k$, a subgraph $H$ of $G$, and $i$ with $1\le i\le k$, define $\pi_i(H)$ to be the subgraph of $G_i$ with
\begin{itemize}
\item vertices $v\in G_i$ such that $v$ is the $i$-th coordinate of a vertex in $H$, and
\item edges $v\!\sim\!v'$ where $v \neq v'$ are the $i$-th coordinates of vertices connected by an edge in $H$ that is derived from an edge in $G_i$.
\end{itemize}
Define the \emph{projection tuple} of $H$ to be the ordered tuple $\big(\pi_1(H),\ldots,\pi_k(H)\big)$.
\end{definition}

We use the projection tuple for the foundation, and define the size of the foundation to be $\sum_{i=1}^k\abs{V(\pi_i(C))}$. Since $\abs{V(\pi_i(C))}\le \abs{V(C)}$,
 the size of the foundation is at most $k$ times the size of the crux, establishing (S5). The remainder of the schema, leading to the count in (S9), depends on peculiarities of the graphs, so we now attack $P_q^k$ and $C_q^k$, in turn. For $P_q^k$, we end up with the finite depth function $q$.

\begin{theorem}
Fix $k$, and let $G_q = P_q^k$. Then $\chi_{G_q}(n)$ is a two-level polynomial of degree function $q^k$ and depth function $q$.
\end{theorem}

\begin{proof}
Given a crux (bond) $H$, let $H_i=\pi_i(H)$, and let the foundation be the projection tuple $(H_1,\ldots,H_k)$, as in Definition~\ref{def:proj}. Then  $H_i$ is a disjoint union of paths (including possibly isolated vertices), that is,
\[H_i \cong P_{c_{i1}}\sqcup \cdots \sqcup P_{c_{is_i}}.\] 
Let $m_i=\sum_j c_{ij} = \abs{V(H_i)}$, so that $\sum_i m_i$ is the foundation size. Isomorphisms of foundations will be tuples $(\sigma_1,\ldots,\sigma_k)$, where $\sigma_i$ is an \emph{order-preserving} embedding of $H_i$ into $P_q$ (that is, the paths $P_{c_{ij}}$ may be placed anywhere in $P_q$ such that they remain in the same order and don't overlap). For a fixed foundation size $M$, there are only a finite number of ways to write $M=\sum_i m_i$, and only a finite number of ways to partition $m_i$ vertices into disjoint paths, so the number of foundation isomorphism classes of  size $M$ is finite (S6). Isomorphisms of foundations induce isomorphisms of cruxes with the appropriate properties (S7) by sending the vertex $(v_1,\ldots,v_k)$ to $\big(\sigma_1(v_1),\ldots,\sigma_k(v_k)\big)$. Any crux mapping to a foundation $(H_1,\ldots,H_k)$ must be a subgraph of $H_1\times\cdots\times H_k$ (with additional properties), and so there are a finite number of crux types for each foundation class (S8).

Lastly, we calculate the number of foundation isomorphisms with domain $(H_1,\ldots,H_k)$ as a function of $q$. The number of isomorphic copies of $H_i$ in $P_q$ is 
 $\binom{q-m_i+s_i}{s_i}$, since we have $q-m_i$ vertices of $P_q$ not included, and so $q-m_i+s_i$ slots to place the $P_{c_{ij}}$'s. Therefore the total number of isomorphisms with domain the given foundation is
\begin{equation}
\label{eqn:PqxPq}
\prod_{i=1}^k\binom{q-m_i+s_i}{s_i} = \frac{1}{s_1!\cdots s_k!}\prod_{i=1}^k (q-m_i+s_i)_{s_i},
\end{equation}
and these are polynomials valid when $q\ge \max_i (m_i-s_i)$.

Careful depth calculations will now establish (S9). Note that adding an edge to the crux increases the codegree of its contribution by one if it does not complete a cycle or by zero if it does complete a cycle, and so the number of edges in a forest contained in the crux is a lower bound on the codegree. If we examine a subgraph consisting of one edge in the crux projecting to each edge in each $H_i$, we see that this subgraph has no cycles (the projection of a cycle onto a path has to repeat edges). In particular, since $H_i$ has $m_i-s_i$ edges, the codegree $r$ of this term is at least
\[\sum_{i=1}^k m_i-s_i\ge \max_i(m_i-s_i),\] meaning that if $e(q):=q\ge r$, the count (\ref{eqn:PqxPq}) is a valid polynomial. This completes (S9) and the proof that it is a two-level polynomial of depth function $q$. Importantly, these polynomials are generally \emph{not} valid for $0\le q<\max_i (m_i-s_i)$, because they may not evaluate to zero in this range; hence the finite depth.
\end{proof}

We similarly analyze $C_q^k$, for which we get slightly smaller depth: $q-2$ rather than $q$.

\begin{theorem}
Fix $k$, and let $G_q = C_q^k$. Then $\chi_{G_q}(n)$ is a two-level polynomial of degree function $q^k$ and depth function $q-2$.
\end{theorem}

\begin{proof}
We follow the proof for $P_q^k$. The only difference is that there are two cases for $H_i=\pi_i(H)$:
\begin{itemize}
\item $H_i$ is a disjoint union of paths, isomorphic to $P_{c_{i1}}\sqcup \cdots \sqcup P_{c_{is_i}},$ or
\item $H_i=C_m$, for some $m$.
\end{itemize}

In the first case, we have a slightly different count for the number of foundation isomorphisms: because of the rotational symmetry of $C_q$: the position of $P_{c_{i1}}$ can be chosen in $q$ ways, and then the remaining positions can be chosen in
\[\binom{q-(m_i-c_{i1})+(s_i-1)}{s_i-1}\]
ways. This is a polynomial valid when
\[q\ge (m_i-c_{i1})-(s_i-1)=(m_i-s_i)-(c_{i1}-1),\]
so in particular, if $e(q):=q-2\ge r$, then 
\[q\ge r+2\ge (m_i-s_i)+2 \ge (m_i-s_i)-(c_{i1}-1),\]
and so the polynomial is valid, establishing (S9) in this case.

In the second case, the number of isomorphisms is $q\cdot\delta_{mq}$, which is not a polynomial in $q$. But the smallest codegree that such a term could contribute is $q-1$ (each edge in $C_q$ increases the codegree by one except for the last one, which completes the cycle), meaning our two-level polynomial will still have depth function $q-2$.
\end{proof}

\begin{remark}
\label{rmk:Biggs}
From limited computer experiments, $\chi_{P_q\times P_q}(n)$ seems to have depth significantly greater than $q$. In contrast, note that $q-2$ is exactly the depth function of
\[\chi_{C_q}(n)=(n-1)^q+(-1)^q(n-1).\]
In fact, $\{K_k\times C_q\}_{k\ge 1}$ form an infinite family of graphs for which $\chi_{K_k\times C_q}(n)$ is a two-level polynomial of degree function $kq$ but still depth function only $q-2$. This is a consequence of \cite[Section 5]{Bi}, which shows that
\[ \chi_{K_k\times C_q}(n)=c_0(n)^q + (n-1)\big(-c_1(n)+(k-1)c_2(n)\big)^q+\text{lower degree terms},\]
where
\[c_i(n)=\sum_{r=0}^{k-i}(-1)^r\binom{k-i}{r}(n-r-i)_{k-r-i}\]
is a polynomial in $n$ of degree $k-1$. In particular, $c_0(n)^q$ is a two-level polynomial of degree function $kq$ and infinite depth, and the next term $(n-1)\big(-c_1(n)+(k-1)c_2(n)\big)^q$ has degree function $1+(k-1)q$, that is, it contributes in codegree $q-1$. In other words, $c_0(n)^q$ is the only contribution to $\chi_{K_k\times C_q}(n)$ of codegree at most $q-2$, and therefore $\chi_{K_k\times C_q}(n)$ has depth function $q-2$.
\end{remark}

\begin{remark}
  Many similar graphs could be proved to have two-level chromatic polynomials of depth function $q$ or $q-2$: $P_q^{k_1}\times C_q^{k_2}\times K_q^{k_3}\times \Kn(q,k_5)^{k_4}\times H$ where $k_j$ are constants and $H$ is a fixed graph, and so forth. Similarly, the other graph polynomials discussed in Section~\ref{sect:chromatic} (characteristic polynomials and matching polynomials) have similar results, as do other products (tensor products, strong products). We leave the reader to fill in the details.
\end{remark}

\section{Two-level polynomials via generating functions}
\label{sect:gfs}
Two-level polynomials occur throughout enumerative combinatorics, and they can sometimes be identified directly from the form of their generating functions. We provide two-examples here, and it would be nice to have a more general understanding of how particular types of generating functions produce two-level polynomials. We note that even our two examples here appear radically different, though: the first is an exponential generating function summed over $q$, as a function of $n$, and the second is an ordinary generating function summed over $n$, as a function of $q$.

\subsection{Sheffer sequences}
Sheffer sequences include binomial type sequences and Appell sequences (some example being falling factorials, Bernoulli polynomials, Hermite polynomials, and Touchard polynomials), and can be defined via generating functions \cite{RKO}, as follows. Let $a(t)=a_0+a_1t+a_2t^2+\cdots$ and $b(t)=b_0+b_1t+b_2t^2+\cdots$ be formal power series over $\QQ$ (or any field of characteristic zero), such that $a_0\ne 0$, $b_0=0$, $b_1\ne 0$. Define the Sheffer sequence $f_q(n)$ via its generating function
\[F(n,t):=\sum_{q\ge 0} \frac{f_q(n)}{q!}t^q = a(t)e^{nb(t)}.\]

\begin{theorem} As defined above, $f_q(n)/b_1^q$ is a two-level polynomial of degree function $q$ and infinite depth.
\end{theorem}

\begin{proof}
Expanding $f_q(n)=\sum_{k\ge 0}c_k(q)n^k$ as a formal power series (for $c_k(q)\in\QQ$), and writing both expressions for $F(n,t)$ as power series in $n$, we have
\[\sum_{k\ge 0}\left(\sum_{q\ge 0}\frac{c_k(q)t^q}{q!}\right)n^k = \sum_{k\ge 0}\frac{a(t)\big(b(t)\big)^k}{k!}n^k.\]
Equating powers of $n$ and examining the $t^q$ coefficient, we have
\[c_k(q) = q! \left\{\frac{ab^k}{k!}\right\}_q,\]
where $\{g\}_q$ means the coefficient of $t^q$ in $g$. Since $a_0\ne 0$, $b_0=0$, $b_1\ne 0$, we see that $c_q(q)\ne 0$ and $c_k(q)=0$ for $k>q$, that is, $f_q(n)$ is a polynomial of degree function $q$.

We want to show that for each codegree $r$, $c_{q-r}(q)/b_1^q$ is a polynomial in $q$, for all $q\ge 0$. Indeed
\[c_{q-r}(q)=q! \left\{\frac{ab^{q-r}}{(q-r)!}\right\}_q = (q)_r\left\{ab^{q-r}\right\}_q,\]
and note that the polynomial $(q)_r$ evaluates to zero if $q<r$ (as desired for infinite depth). Now $\left\{ab^{q-r}\right\}_q$ is the sum, over all compositions
\[q=c+(d_1+1)+\cdots+(d_{q-r}+1)\]
with $c,d_i$ nonnegative integers, of $a_cb_{d_1+1}\cdots b_{d_{q-r}+1}$. Since $c+d_1+\cdots+d_{q-r}=r$, a bounded number of these $d_i$ are positive; call them $\vec e=(e_1,\ldots,e_m)$. Inspired by our schema, our objects are the compositions, and define both the cruxes and foundations to be $(c,\vec e)$. Then the \emph{types} are indexed by $T=\big(c,\text{ multiset}\{e_i\}\big)$, and there are a fixed number of types for a given $r$.

Each composition of  type $T$ contributes
\[a_cb_{e_1+1}\cdots b_{e_m+1}\cdot b_1^{q-r-m} = \frac{a_cb_{e_1+1}\cdots b_{e_m+1}}{b_1^{r+m}}\cdot b_1^q = h_T \cdot b_1^q,\]
where $h_T$ is a constant. The number of compositions of type $T$ is the number of ways to order $e_1,\ldots,e_m$ amidst $q-r-m$ 1's (corresponding to $d_i=1$ in the composition of $q$), which, provided $r\le q$, is
\[s_T(q):=\frac{(q-r)_m}{\lambda_1!\cdots\lambda_\ell!},\]
where $\lambda_j$ are the multiplicities in the multiset $\{e_i\}$.

Putting this all together, we see that
\[\frac{c_{q-r}(q)}{b_1^q} = (q)_r \left\{\frac{ab^k}{k!}\right\}_q\cdot\frac{1}{b_1^q}=(q)_r\sum_{T} s_T(q)h_T,\]
which is a polynomial! Note that the polynomial $s_T(q)$ is invalid in this count for $0\le q<r$, but in that case it is cancelled by $(q)_r=0$.
\end{proof}

\subsection{Partitions of $n$ into $q$ parts}
\begin{theorem}
Let $f_q(n)$ be the number of partitions of $n$ into $q$ parts. Then $q!(q-1)!f_q(n)$ is a two-level quasi-polynomial of degree function $q-1$, depth function $\floor{(q-1)/2}$, and period function one. 
\end{theorem}

\begin{proof}
Beck, Gessel, and Komatsu \cite{BGK} examine a more generic question, Cimpoea\c{s} \cite{Cim} tackles this question, and both get us half-way there, as follows. Transposing the partition diagram for such a partition, we have a partition of $n$ with largest part exactly $q$. Now we can compute the generating function $\sum_{n=0}^\infty f_q(n)x^n$ 
\begin{align*} &= (1+x+x^2+\cdots)(1+x^2+x^4+\cdots)\cdots(1+x^{q-1}+x^{2(q-1)}+\cdots)(x^q+x^{2q}+\cdots)\\
&=\frac{x^q}{(1-x)(1-x^2)\cdots(1-x^q)}.
\end{align*}

Let $r_1=1, r_2=-1,r_3,\ldots,r_a$ be all of the $k$th roots of unity, over all $1\le k\le q$. The multiplicity $m_i$ of $r_i$ in the denominator of the generating function is the number of $k$ ($1\le k\le q)$ such that $r_i$ is a $k$th root of unity. In particular, $m_1=q$, $m_2=\floor{q/2}$, and $m_i\le m_2$ for $i\ge 3$. Expanding the generating function using partial fractions, we get that
\[f_q(n) = \sum_{i=1}^a P_{q,i}(n)r_i^n,\]
where $P_{q,i}(n)$ is a polynomial in $n$ of degree $m_i-1$. In particular, $f_q(n)$ is a quasi-polynomial of period $\lcm(1,\ldots,q)$ and degree function $m_1-1=q-1$, and the terms of codegree at most  
\[m_1-m_2-1=q-\floor{q/2}-1=\floor{(q-1)/2}\]
 are determined by $P_{q,1}(n)$. This $P_{q,1}(n)$ is called the \emph{polynomial part} of $p_q(n)$, and we will prove that $q!(q-1)!P_{q,1}(n)$ is a two-level polynomial of infinite depth, from which the proposition follows.
 
 Both \cite{BGK} and \cite{Cim} give the formula:
 \[P_{q,1}(n) = \frac{1}{q!}\sum_{u=0}^{q-1}\frac{(-1)^u}{(q-1-u)!}
\sum_{i_1+\cdots +i_q=u}\frac{B_{i_1}\cdots B_{i_q}}{i_1!\cdots i_q!}1^{i_1}\cdots q^{i_q}(n-q)^{q-1-u}, \]
where $B_i$ are the \emph{Bernoulli numbers}. Therefore
\[q!(q-1)!P_{q,1}(n) = \sum_{u=0}^{q-1}(-1)^u (q-1)_u (n-q)^{q-1-u}
\sum_{i_1+\cdots +i_q=u}\frac{B_{i_1}\cdots B_{i_q}}{i_1!\cdots i_q!}1^{i_1}\cdots q^{i_q}.\]
By our algebra of two-level polynomials in Section~\ref{sect:alg}, it suffices to prove that, for fixed $u$,
\[\sum_{i_1+\cdots +i_q=u}\frac{B_{i_1}\cdots B_{i_q}}{i_1!\cdots i_q!}1^{i_1}\cdots q^{i_q}\]
is a polynomial in $q$, valid for all $q\ge 1$. Since $B_0=1$, $0!=1$, and $k^0=1$, only the nonzero $i_j$'s matter in the sum. The nonzero $i_j$'s correspond (by reordering from smallest to largest) to some partition $\lambda$ of the integer $u$. Since $u$ is fixed, there are a fixed number of possible partitions $\lambda$. Therefore it suffices to prove that the sum over all $i_1, \ldots, i_q$ corresponding to a fixed $\lambda$ is a polynomial in $q$. For such terms, $\frac{B_{i_1}\cdots B_{i_q}}{i_1!\cdots i_q!}$ is a constant, and so it suffices to prove that the sum of $1^{i_1}\cdots q^{i_q}$ over all $i_1,\ldots,i_q$ corresponding to a fixed $\lambda$ is a polynomial in $q$. This sum is the \emph{monomial symmetric function} $m_\lambda$, evaluated at $(1,2,\ldots,q,0,0,\ldots)$. All symmetric functions may be written as sums and products of the \emph{elementary} symmetric functions \cite[Theorem 7.4.4]{Stanley2},
\[e_r(X_1,X_2,\ldots)=\sum_{1\le j_1<j_2<\cdots<j_r}X_{j_1}X_{j_2}\cdots X_{j_r},\]
so it suffices to prove that $e_r(1,2,\ldots,q,0,0,\ldots)$ is a polynomial in $q$, for fixed $r$. But this is the  coefficient of $n^{q-r}$ in $n(n+1)\cdots (n+q)$, which is a polynomial in $q$ by Example~\ref{ex:rising_fact}. The proof follows.
\end{proof}

\begin{remark} We have seen that $q!(q-1)!f_q(n)$ is a two-level quasi-polynomial of period function one, which is slightly odd, considering that for fixed $q$ the period of the quasi-polynomial $f_q(n)$ is actually $\lcm(1,2,\ldots,q)$. The discrepancy comes from the fact that $q!(q-1)!f_q(n)$ has finite depth function $\floor{(q-1)/2}$, and the leading $\floor{(q-1)/2}$ terms do indeed have period one.
\end{remark}

\section{Nonattacking knights}
\label{sec:knights}
Suppose we have an $n\times n$ chessboard with $q$ identical pieces, and each piece has the same \emph{finite} set of attacking moves $M\subseteq \ZZ\times\ZZ$. In particular, if we denote piece  $i$'s placement on the chessboard by $\vec x_i=(x_i,y_i)\in [n]^2$, then we say that piece $i$ attacks piece $j$ if $\vec x_j - \vec x_i\in M$. For example, the attacking moves for a knight are the nine vectors $(\pm 1, \pm 2)$, $(\pm 2, \pm 1)$, and $(0,0)$ (to rule out being on the same square). This is in contrast to the nonattacking queens problem, where pieces could move arbitrarily far in certain directions.

We label the $q$ identical pieces by $1,2,\ldots,q$, and let $f_q(n)$ be the number of ways to place the labeled pieces on the $n\times n$ board such that none of them are attacking each other; the traditional count would now divide by $q!$ to count unlabeled pieces. We want to show that $f_q(n)$ is a two-level polynomial, but that will not quite be true: it's only true for $n$ sufficiently large. An easy fix would be to define an ``eventual'' two-level polynomial in the obvious way: by requiring that each $f_q$ (fixed $q$) agree with the polynomial for sufficiently large $n$. We do this in the following definition:

\begin{definition}
\label{def:eventual-two-level}
Let $g$ be an eventually nonnegative numerical polynomial, let $S=\{q\in\NN:\ g(q)\ge 0\}$, and let $e:S \to \NN \cup \{-1,\infty\}$.

\begin{enumerate}
\item[(a)]  An \emph{eventual two-level polynomial} of degree function $g$ and depth function $e$ is an infinite sequence $\{f_q\}_{q\in S}$ of functions that eventually (for sufficiently large $n$) agree with a polynomial of degree $g(q)$, that has the following property: writing
\[f_q(n) = \sum_{r=0}^{\infty}c_r(q)n^{g(q)-r}\]
for sufficiently large $n$, where $c_r(q)=0$ if $r>g(q)$, then there exist polynomials $\{\phi_r(q)\}_{r\ge 0}$ such that, for every $q\in S$ and every $r \leq e(q)$, we have $c_r(q) = \phi_r(q)$. 
 \item[(b)] As before, for $q\in \NN\setminus S$, define $f_q(n)=0$ and $e(q)=-1$.
 \end{enumerate}
 \end{definition}

This is a fine definition, but a bit weak. As we shall see, the codegree $r$ term in the sum will be ``valid'' for $n\ge \ell(r)$, where $\ell(r)$ is a linear function of $r$. This is a much stronger statement, because it quantifies much better what ``eventually'' means and that for low codegree terms ``eventually'' kicks in very soon. But it is also a nonsensical statement, because, for a given $n$,  $f_q(n)$ either agrees with a given polynomial or it doesn't; that is, there is no way to say that the leading terms of $f_q(n)$ agree with a polynomial while other terms don't.

Instead, it seems that we have to write $f_q(n)$ as an explicit sum of functions that are eventually polynomial, and then note that a codegree $r$ term in the sum agrees with a polynomial for $n\ge \ell(r)$. We will derive a formula for $f_q(n)$, and then make a precise statement at the end.

We follow the schema, with the proviso that we only have \emph{eventual} polynomials in $n$, which is relevant in (S1) and (S2). Given placements $\vec x_i\in [n]^2$ of each piece ($1\le i\le q$), we define the \emph{attack graph} of this placement to have vertices $[q]$ and directed edges $i\rightarrow j$ if piece $i$ is attacking piece $j$, and we label the edge by the attacking move $m\in M$ with $\vec x_j - \vec x_i = m$. The attack graphs that are realizable by some piece placement form a lattice $\mathcal L_q$ under (label-respecting) inclusion, and so
\begin{equation}
\label{eqn:knights_sum}
f_q(n) = \sum_{G\in \mathcal L_q}\mu_{\mathcal L_q}\big([\hat{0},G]\big)f^{G}_q(n),
\end{equation}
where $f^{G}_q(n)$ is the number of placements whose attack graph \emph{contains} $G$. Let's calculate $f^{G}_q(n)$ more precisely, which will give us (S1) and (S2), for sufficiently large $n$.

Note that we can analyze each connected component of the underlying undirected graph of $G$ separately, and the final answer is simply the product of these: since there are no attacking edges between two components, they may be placed independently; we are not concerned with whether placements allow additional attacking moves, only that $G$ is a subgraph of all of the attacks. So let $C$ be a connected component of $G$. If $C$ is an isolated vertex, then the number of placements of that piece is $n^2$. Otherwise,  let $i_0$ be one of the pieces in C, and we see that every piece in $C$'s placement is determined relative to piece $i_0$. That is, there exist relative positions $\vec p_i\in \ZZ^2$ such that $\vec x_i = \vec x_{i_0}+\vec p_i$ must be true for every $i$. Let
\[a=\max_{i,j}\{p_{i1}-p_{j1}\},\quad b=\max_{i,j}\{p_{i2}-p_{j2}\}.\]
When the pieces are placed on the board, the difference in the $x$ coordinates of the furthest right and furthest left pieces will be exactly $a$, and the difference in the $y$ coordinates between the furthest up and furthest down pieces will be exactly $b$. In other words, the number of ways to place the  pieces  of $C$ in this attacking configuration is
\[(n-a)(n-b),\]
assuming $n\ge\max\{a,b\}$. However, for $n<\max\{a,b\}$, the number of attack configurations is zero, but $(n-a)(n-b)$ may evaluate to something nonzero.

Let $w$ be the largest absolute value of any coordinate of any element of $M$ ($w=2$ for knights). Then the largest $a$ or $b$ could be is $w(k-1)$, where $k$ is the number of pieces in $C$, achieved when $C$ is a directed path with each piece attacking the next along the move corresponding to $w$.

For graphs $G$ with multiple components, we put this all together now. Given $\vec a = (a_1,\ldots,a_d)\in \NN^d$, define the function
\[g_d(\vec a; n)= \begin{cases} (n-a_1)\cdots(n-a_d), & \text{if $n\ge\max a_i$,}\\
0, & \text{else.} \end{cases}\]
Suppose $G$ has $c$ components with more than one vertex, with $k_1,\ldots,k_c\ge 2$ their respective numbers of vertices. Let $s=k_1+\cdots+k_c$, and so there are $q-s$ isolated vertices in the attack graph. Then there exists $\vec a\in \NN^{2c}$ such that the number of placements whose attack graph contains $G$ is
\begin{equation}\label{knight_count}
f^{G}_q(n) = n^{2(q-s)}g_{2c}(\vec a; n),
\end{equation}
where $\max{a_i}\le w\max(k_i-1)$. Furthermore, the codegree $r$ of this term (with respect to the overall degree $2q$) is
\[r=2q - 2(q-s) - 2c = 2(s-c) = 2\sum_{i=1}^c(k_i-1) \ge 2\max(k_i-1).\]
In particular,
\[\max{a_i}\le w\max(k_i-1) \le wr/2,\]
and so if $f^{G}_q(n)$ has codegree $r$, then it agrees with
\[n^{2(q-s)}(n-a_1)\cdots(n-a_{2c}),\]
a two-level polynomial of degree function $2q-r$ and infinite depth, as long as
\[ n \ge wr/2.\]

We are ready to continue with the schema. The objects are attack graphs, the cruxes are unions of the components that are not isolated vertices, and the crux isomorphisms are (labeled, directed) graph isomorphisms. Combining Equations (\ref{eqn:knights_sum}) and (\ref{knight_count}) gives us (S1) and (S2), valid for $n$ sufficiently large. The maximum degree is $2q$ from the attack graph with empty crux (S3). Using the above notation, the size of a crux is $s$, and
\[r=2(s - c)= s + (k_1+\cdots+k_c -2c) \ge s,\]
because $k_i\ge 2$, establishing (S4). Let the foundation for a crux be the set of vertices (pieces); its size is the same as the size of the crux (S5). The isomorphisms of foundations are set bijections (S6), inducing crux isomorphisms (S7) by renaming the vertices. Cruxes on a fixed set of vertices are (labeled, directed) graphs on these vertices (with additional properties) and so the number of them is finite (S8), and the number of foundation isomorphisms (S9) with domain $[k]$ is $(q)_k$.

Using the schema and Lemma~\ref{lem:automorph}, we can write
\begin{equation}\label{eqn:knights_final}
f_q(n) = n^{2q} +  \sum_{\substack{r\ge 2,\\r\text{ even}}} \sum_{t=1}^{m_r} \frac{(q)_{k_{r,t}}}{C_{r,t}}n^{2q-2c_{r,t}-r}g_{2c_{r,t}}(\vec a_{r,t}; n),
\end{equation}
where $\max_i\left(\vec a_{r,t}\right)_i \le wr/2$.

We summarize this analysis as a theorem.

\begin{theorem}
Let $f_q(n)$ be defined as in this section. Then $f_q(n)$ is an eventual two-level polynomial of degree function $2q$ and infinite depth. Furthermore, Equation~(\ref{eqn:knights_final}) writes $f_q(n)$ as a sum whose codegree $r$ term agrees with a two-level polynomial (of degree function $2q-r$ and infinite depth) for $n\ge wr/2$.
\end{theorem}

\begin{remark} Note that for fixed $q$, $f_q(n)$ is a polynomial, rather than merely a quasi-polynomial, unlike for the $q$-queens problem.
\end{remark}

\begin{remark}
These  functions $g_d(\vec a;n)$ could have been used in Section~\ref{sect:finite}, with $q$ rather than $n$, to get infinite sums of explicit expressions that are valid for all $q\in\NN$, rather than only for $q$ such that $e(q)\ge r$. In particular, in infinite depth examples, (S9) arose from counts such as $(q)_d$, and the key to infinite depth is that the equality
\[(q)_d = g_d\big((0,1,\ldots,d-1); q\big)\]
holds for all $q\ge 0$. In the finite depth case, we needed formulas like $(q-a)_d$, which were only valid for $q\ge a$. But the formula
\[g_d\big((a,a+1,\ldots,a+d-1); q\big)\]
would have been valid for all $q\ge 0$ because it evaluates to zero if $q<a$, by definition.
\end{remark}

\begin{remark} We phrased this problem in terms of explicit attack graphs, rather than the more generic hyperplane arrangements of Section~\ref{sect:Ehrhart}. There is an underlying subspace arrangement, but it is \emph{affine} rather than \emph{linear}. That is, an attack constraint such as $\vec x_j - \vec x_i = (2,1)$ is a flat, but it does not go through the origin. This is why the formulas are not correct for small $n$: the polynomial counts will only be valid once the flats intersect the set of placement configurations, $[n]^{2q}$. For the $q$-knights problem, we needed to keep careful track of when they begin to intersect, so that we could state explicitly for which $n$ they agree with a polynomial.
\end{remark}

\section{Future Directions}
\label{sec:conc}

We have now established that a wide variety of counting functions are two-level polynomials, uncovering a hidden structure unifying different types of combinatorial objects. A natural next step would be to look for other examples of two-level polynomials, and we believe that the schema we have described will be useful in the search.

Graph polynomials are one natural place to keep looking for two-level polynomials. We showed that some of these (chromatic, characteristic, matching polynomials) yield two level-polynomials when applied to interesting families of graphs, $\{G_q\}_{q\in\NN}$. Others (flow polynomials) seem promising, but their status remains open. And there are many more graph polynomials that we have barely examined \cite{EM1, EM2}. Many of these are specializations of the bivariate Tutte polynomial of the graph $G=(V,E)$,
\[T_G(m,n) = \sum_{A\subseteq E} (m-1)^{k(A)-k(E)}(n-1)^{k(A)+\abs{A}-\abs{V}},\]
where $k(A)$ is the number of connected components of the graph $(V,A)$. This raises the question of whether there is some generalization of two-level polynomials to multiple ``n'' variables, and whether it applies to Tutte polynomials of certain families. Though what that means is unclear: given some $f_q(m,n)$, even the notion of codegree is now much more subtle. On the other hand, functions like
\[f_{q,s}(n)=\chi_{P_q\times P_s}(n)\]
with two ``$q$'' variables seem ripe for analysis using a variation of our schema, and we believe they would have clearly-definable two-level polynomial behavior.

The \emph{transfer-matrix method} \cite[Section 4.7]{Stanley1} is likely a fruitful source of two-level polynomials. Transfer-matrices can encode complicated recursions, and Engstr\"om and Kohl \cite{EK}, for example, use them to calculate chromatic polynomials for many families of graphs. The counting functions end up being computed via matrices of the form $(M_n)^q$, where the entries of $M_n$ are polynomials in $n$. As a concrete example, by recursively keeping track of whether the last copy of $P_3$ in $P_3\times P_q$ is colored with two or three colors, we were able to calculate that
\[\chi_{P_3\times P_q}(n)=
\begin{bmatrix} n^3-3n^2+2n & n^2-n\end{bmatrix}
\begin{bmatrix}
n^3-6n^2+14n-13 & n^2-4n+5 \\
n^3-6n^2+13n-10 & n^2-3n+3
\end{bmatrix}^{q-1}
\begin{bmatrix} 1 \\ 1 \end{bmatrix}.\]
This raises questions such as: When is $(M_n)^q$ a two-level polynomial? What can we say about its depth? What interesting examples does this give us?

Biggs \cite{Bi} used these methods to examine the chromatic polynomial for $K_k\times C_q$, for fixed $k$. In this particular example, the eigenvalues (and their multiplicities) of an associated transfer-matrix turn out to be polynomials in $n$, allowing us to compute
\[ \chi_{K_k\times C_q}(n)=c_0(n)^q + (n-1)\big(-c_1(n)+(k-1)c_2(n)\big)^q+\text{lower degree terms},\]
in the notation from Remark~\ref{rmk:Biggs}. In general, an eigenvalue expansion form is very useful for understanding the asymptotics in $q$, and the two-level polynomial expansion is very useful for understanding the asymptotics in $n$, and it would be nice to understand the connection between the two in greater generality.

In Section~\ref{sect:gfs}, we gave two very different generating function arguments for counting functions being two-level polynomial, one using an exponential generating function summed over $q$, as a function of $n$, and another using an ordinary generating function summed over $n$, as a function of $q$. It would be nice to generalize either of these methods, or even find a unifying framework. We note that generating function methods are often closely tied to the transfer-matrix method.

The $q$-analog of many counting functions seem to produce candidate functions for two-level polynomials. For example, for fixed $k\in\NN$, the Gaussian binomial --- and confusingly, we will have to switch the traditional roles of $q$ and $n$ to match our notation ---  is defined to be
\[f_q(n)=\begin{bmatrix} q \\ k \end{bmatrix}_n := \frac{\left(1-n^q\right)\left(1-n^{q-1}\right)\cdots \left(1-n^{q-k+1}\right)}{\left(1-n^k\right)\left(1-n^{k-1}\right)\cdots \left(1-n\right)},\]
and one can check that it is a two-level polynomial of degree $k(q-k)$ and depth $q-k$. There are many $q$-analogs \cite{Stanley1} one could examine to look for two-level polynomial behavior.

Once you know that something is counted by a two-level polynomial, what do you do with that information? The most obvious thing to do is to try to calculate several of the leading coefficients, as functions of $q$. These coefficients give you more and more specific information about the asymptotics as $n\rightarrow \infty$. With luck, you might notice interesting patterns in these coefficients, or possibly even compute the entire function. In their collective seven-paper series on placing $q$ non-attacking chess pieces on an $n\times n$ board \cite{CHZ, CHZ2, CHZ3, CHZ4, CHZ5, CHZ6, CHZ7}, Chaiken, Hanusa, and Zaslavsky compute several coefficients for chess (and chess-like) pieces that can move arbitrarily far in specified directions. A natural place to attempt calculations, then, is to examine the $q$-knights problem (Section~\ref{sec:knights}) using similar methods. The long history and usefulness of graph polynomials mean that these are also an inviting place to attempt calculations.

Beyond simply calculating coefficients, the depth and period functions remain a bit mysterious, especially because their values are often better than the theory predicts. For example, Ehrhart theory predicts that the $q$-bishops counting function will have coefficient period functions $p(r)$ exponential in $r$. And yet \cite{CHZ6} establishes that we may simply take $p(r)=2$ for all $r$. Investigating this further seems intriguing. Similarly, experimental evidence shows that the chromatic number of $P_3\times P_q$ (and similar graph families) has much greater \emph{depth} than expected. For example, the theory guarantees that $e(q)=q$ works for the depth, but $e(q)=\infty$ for $q\le 4$ and $e(5)=14$ are the best possible depths (note that, when $q=5$, the degree is 15; the constant term does not agree with the codegree 15 polynomial). A better understanding may require more subtle combinatorial techniques. This current paper is the first time that we believe the concept of depth has been formally defined, and there is still much to understand about it. For example, the theory predicts the same depth for all of the different graph polynomials that we've proven anything about, though experimentally this is unclear. Are the depths the same?

Above all, we see the two-level polynomial idea as creating a common language for ideas that have popped up in many areas of combinatorics. We hope that the formalization and tools provided here can be a starting point for investigating all of these questions.

\section*{Acknowledgements} Tristram Bogart was supported by internal research grant INV-2025-213-3438 from the Faculty of Sciences of the Universidad de los Andes.

\end{document}